\documentclass[a4paper,twoside]{article}

\usepackage{lmodern}
\usepackage[T1]{fontenc}
\usepackage{inputenc}
\usepackage{amsmath}
\usepackage{amsthm}
\usepackage{amsfonts}
\usepackage{graphicx}
\usepackage{color}
\usepackage[all]{xy}
\usepackage{multicol}
\usepackage{amssymb}
\usepackage{setspace}
\usepackage{makeidx}
\usepackage{fancybox, calc}
\usepackage{tikz}
\usepackage{eso-pic}
\usepackage{pdflscape}
\usepackage[some]{background}
\usepackage{enumitem}
\usepackage{tikz}

\theoremstyle{definition}
\newtheorem{defin}{Definition}
\newtheorem{ejem}{Example}

\theoremstyle{definition}
\newtheorem{teo}{Theorem}
\newtheorem{cor}{Corollary}
\newtheorem{prop}{Proposition}
\newtheorem{lema}{Lemma}

\theoremstyle{remark}
\newtheorem{rem}{Remark}

\newcommand\fun{\section*{Fundings}}
\newcommand\acks{\section*{Acknowledgements}}
\makeindex

\title{Global Invariant Branches of Non-degenerate Foliations on Projective Toric Surfaces}
\author{Beatriz Molina-Samper}

\date{}

\begin{document}
\maketitle

\begin{abstract}
	We show that the isolated invariant branches globalize to algebraic curves, when we consider weak toric type complex hyperbolic foliations on projective toric ambient surfaces. To do it, we pass through a characterization of weak toric type foliations in terms of ``non-degeneracy'' conditions, associated to Newton polygons. We also give a description of the relationship between invariant algebraic curves and isolated invariant branches, valid for the case of toric type, by means of the following dichotomy. Either there is a rational first integral and there are no isolated invariant branches or we have only finitely many global invariant curves, all of them extending isolated invariant branches.
\end{abstract}
\maketitle
\section{Introduction}
The aim of this paper is to describe local-global features for the invariant curves of weak toric type complex hyperbolic foliations on projective toric ambient surfaces. The main result we present is stated as follows:
\begin{quote}
	\textbf{Theorem \ref{teo:prolongacion}.} The isolated invariant branches of a complex hyperbolic weak toric type foliation on a projective toric surface extend to projective algebraic curves.
\end{quote}
A foliation is of toric type when it has a combinatorial desingularization. This definition was introduced by M.I.T. Camacho and F.Cano in \cite{Cam-C}. Analogously, a foliation is of weak toric type when it has a combinatorial desingularization, but just up to presimple points. We call complex hyperbolic to the foliations without saddle-nodes after reduction of singularities (they are also called ``generalized curves'').
An invariant branch is isolated if it always falls into a non-dicritical component of the exceptional divisor after any reduction of singularities. This concept was suggested in \cite{Cam-N-S} by C. Camacho, A. Lins Neto and P. Sad.

We characterize weak toric type foliations in terms of Newton polygons and ``initial forms''. To do it, we introduce the concept of Newton non-degenerate foliation, following the classical ideas of A.G. Kouchnirenko and M. Oka for varieties, that can be found in \cite{Kus,Oka}. We prove the equivalence result below:
\begin{quote}
	\textbf{Theorem \ref{teo:equivalencia}:} A complex hyperbolic foliated surface is Newton non-degenerate if and only if it is of weak toric type.
\end{quote}
A foliated surface is the data of a foliation $\mathcal{F}$ on a complex surface $M$ and a normal crossings divisor $E \subset M$. Most of the definitions and properties we present in this paper concern to the pair $(\mathcal{F},E)$ and not only to the foliation $\mathcal{F}$.

Let us recall that a nonsingular projective toric surface is naturally endowed with a normal crossings divisor given by the union of the non-dense orbits of the torus action. Moreover, these surfaces can be obtained by blowing-ups and blowing-downs from the projective plane $\mathbb{P}_{\mathbb{C}}^2$ with the ``standard'' toric structure, that gives the divisor $X_0X_1X_2=0$.
Most of the properties we are going to consider are stable under equivariant (combinatorial) blowing-ups and blowing-downs. This allow us to prove many of the results by looking just to the projective plane.

In order to describe Newton non-degenerate foliations on the projective plane, we use in an essential way the following property: ``The number of roots of a Laurent polynomial system in general position is the mixed volume of the associated polyhedra''. This result was proved by D.N. Bernstein, A.G. Khovanskii and A.G. Kouchnirenko in  \cite{Ber, Kho}. Applying it, we show that the homogeneous polygon $\Delta_h(\mathcal{F})$ of a Newton non-degenerate foliation $\mathcal{F}$ is a single vertex or a segment. In this way, we describe a set of projective algebraic curves such that any isolated invariant branch at a given point is the germ of one of these curves at the point. More precisely, we have the following three cases for the homogeneous polygon $\Delta_h(\mathcal{F})$:
\begin{quote}
	\begin{description}
		\item [Case a)] It is a single point: there are no isolated invariant branches.
		\item [Case b)] It is the segment joining the points $(0,d,0)$ and $(0,0,d)$: the isolated invariant branches are in a finite family of lines $\ell_{\lambda}=(X_2-\lambda X_1=0)$.
		\item [Case c)] It is the segment joining the points $(d,0,0)$ and $(0,a,d-a)$: the isolated invariant branches are in a finite family of curves $\mathcal{C}_{\lambda}=(X_1^{\tilde{d}-\tilde{a}}X_2^{\tilde{a}}-\lambda X_0^{\tilde{d}}=0)$.
	\end{description}
\end{quote}
In this way, we obtain the proof of Theorem \ref{teo:prolongacion}.

Concerning the existence of isolated invariant branches in the weak toric type case, we see in Lemma \ref{lema:conexion} that we effectively find at least one for each $\lambda$ in the above families. In case b), we just blow-up the common point of all the lines $\ell_{\lambda}$. Each of the transformed lines cuts the new divisor at two points $p_{\lambda}$ and $q_{\lambda}$ and we prove that the ``eigenvalues ratios'' of the singularities of the foliation at these points are opposite each to the other, hence one of them is a simple point and we find an isolated invariant branch through it. In case c), we find a similar property after reduction of singularities of the cuspidal family $\mathcal{C}_{\lambda}=(X_1^{\tilde{d}-\tilde{a}}X_2^{\tilde{a}}-\lambda X_0^{\tilde{d}}=0)$.

When we have the strongest property that the foliation is of toric type, we obtain the following result:
\begin{quote}
	\textbf{Theorem \ref{teo:dicotomia}:} We have the next dichotomy for a toric type foliation on a projective toric surface:
	\begin{enumerate}
		\item [I)] There is rational first integral and there are no isolated invariant branches.
		\item [II)] There is no rational first integral and every proper invariant branch extending to a projective algebraic curve is an isolated invariant branch.
	\end{enumerate}
\end{quote}
An invariant branch is proper when it is based at a point of the divisor but is not contained in it. 

From the results of this paper, we know that there is always global invariant curve for a toric type foliation on a projective toric surface. This does not hold for general complex hyperbolic foliations on the projective plane. Indeed, Jouanolou's classical example (see \cite{Jou}) given by the differential form
$$
(X_0^2X_1-X_2^{3})dX_0+(X_1^2X_2-X_0^{3})dX_1+(X_2^2X_0-X_1^{3})dX_2
$$
has no algebraic invariant curves, but it has seven singularities and two isolated invariant branches at each. 

Jouanolou's example is classically used to construct germs of codimension one foliations in dimension three without invariant surface. In a forthcoming paper, we apply these results to prove the existence of invariant surface for germs of toric type codimension one foliations in dimension three.

\section{Generalities on Foliated Surfaces}
We introduce basic definitions and results concerning the theory of holomorphic singular foliations in dimension two. These contents can be essentially found at \cite{Can-C-D}.
\subsection{Foliated surfaces}

A \emph{nonsingular complex analytic surface} $M$, is a $\mathbb{C}$-ringed space $M=(|M|, \mathcal{O}_M)$ in local $\mathbb{C}$-algebras of functions, covered by open subsets isomorphic to open subsets of ($\mathbb{C}^2,\mathcal{O}_{\mathbb{C}^2})$. Denote by $\Omega^1_M$ the sheaf of germs of holomorphic one-forms on $M$. A \emph{codimension one holomorphic singular foliation} $\mathcal{F}$ on $M$ (for short, a \emph{foliation} on $M$) is an invertible subsheaf $\mathcal{F}\subset \Omega_M^1$, locally generated at each point $p\in |M|$ by a holomorphic one-form $\omega \in \Omega^1_{M,p}$, that we write in local coordinates as
$$
\omega= f_1dx_1+f_2dx_2,
$$
where $f_1,f_2 \in \mathcal{O}_{M,p}$ have no common factors. The \emph{singular locus} $\text{Sing}(\mathcal{F})$ is the closed analytic subset of $M$ locally defined by $f_1=f_2=0$. It is a set of isolated points.

A \emph{normal crossings divisor} $E$ of $M$ is the union of a finite family $\{E_i\}_{i \in I}$ of connected closed nonsingular holomorphic curves such that, for each point $p \in |M|$ we have $E\subset (x_1x_2=0)$, where $(x_1,x_2)$ is a local coordinate system. Note that the $E_i$ are the irreducible components of $E$. We denote by $e_p(E)$ the number of irreducible components of $E$ through $p \in |M|$, we have that $e_p(E) \in \{0,1,2\}$. We say that $E$ is a \emph{strong normal crossings divisor} if either $E_i \cap E_j$ is empty or it is a single point, for every $i,j$.

Given a point $p \in |M|$, a curve branch $(\Gamma,p)$ is defined by an equation $f=0$, where $f\in \mathcal{O}_{M,p}$ is irreducible. We say that $(\Gamma, p)$ is an \emph{invariant branch of $\mathcal{F}$} if $\omega \wedge df = f\alpha$, where $\alpha$ is a germ of holomorphic 2-form and $\omega$ is a generator of $\mathcal{F}$ at $p$. We know that there is an only invariant branch $(\Gamma,p)$ through $p$, when $p \not\in \text{Sing}(\mathcal{F})$.

Consider an irreducible curve $Y$ of $M$ and a point $p\in Y$. If $(\Gamma,p) \subset (Y,p)$ is an invariant branch, then every branch $(\Upsilon,q) \subset (Y,q)$ is also invariant, for each $q\in Y$. In this case, we say that $Y$ is an \emph{invariant curve of $\mathcal{F}$}. The non-invariant irreducible components of $E$ are also called \emph{dicritical components}. We write the index set as $I=I_{\text{inv}}\cup I_{\text{dic}}$, where $I_{\text{inv}}$ corresponds to the invariant components and $I_{\text{dic}}$ corresponds to the dicritical ones.
We also denote
$$
E_{\text{inv}}=\cup_{i \in I_{\text{inv}}}E_i; \quad E_{\text{dic}}=\cup_{i \in I_{\text{dic}}}E_i.
$$

We say that \emph{$\mathcal{F}$ and $E$ have normal crossings at $p\notin \text{Sing}(\mathcal{F})$} if $E\cup \Gamma$ is a local normal crossings divisor, where $(\Gamma,p)$ is the only invariant branch of $\mathcal{F}$ through $p$. The \emph{adapted singular locus} $\text{Sing} (\mathcal{F},E)$ is defined by
$$
\text{Sing} (\mathcal{F},E)=\text{Sing} (\mathcal{F}) \cup \{p \notin \text{Sing} (\mathcal{F}); \; \mathcal{F} \text{ and } E \text{ have no normal crossings at } p\}.
$$
We have that $\text{Sing} (\mathcal{F},E)$ is a set of isolated points and $\text{Sing} (\mathcal{F},E) \supset \text{Sing} (\mathcal{F})$.
\begin{rem}
	If $|M|$ is a compact set or a germ around a compact set, then $\text{Sing} (\mathcal{F},E)$ is finite.
\end{rem}
We say that a coordinate system $(x_1,x_2)$ at $p \in |M|$ is \emph{adapted to $E$} if $e_p(E)=0$, $E=(x_1=0)$ or $E=(x_1x_2=0)$. 
Let $\Omega^1_M(\log E)$ be the sheaf of germs of logarithmic one-forms along $E$. A \emph{codimension one singular $E$-foliation $\mathcal{L}$ on $M$} (for short, an \emph{$E$-foliation} on $M$) is an invertible subsheaf $\mathcal{L}\subset \Omega^1_M(\log E)$, locally generated at each point $p\in |M|$ by a logarithmic one-form $\eta \in \Omega^1_{M,p}(\log E)$, that we write in adapted local coordinates as
$$
\eta =\sum_{i=1}^ea_i\dfrac{dx_i}{x_i}+\sum_{i=e+1}^{2}a_idx_i; \quad e=e_p(E),
$$
where the coefficients $a_i$ have no common factors. We define the \emph{adapted multiplicity $\nu_p(\mathcal{F},E)$} to be the minimum $\nu_p(a_1,a_2)$ of the orders $\nu_p(a_1)$ and $\nu_p(a_2)$ at $p$ of the coefficients.

Denote by $\text{Fol}(M,E)$ the set of $E$-foliations on $M$ and by $\text{Fol}(M)$ the set of foliations. Observe that $\text{Fol}(M,\emptyset)=\text{Fol}(M)$. There is a bijection between $\text{Fol}(M)$ and $\text{Fol}(M,E)$ given by $\mathcal{F} \mapsto \mathcal{L}_{\mathcal{F}}$, determined by the relation $\omega = x_1^{\varepsilon_1}\cdots x_e^{\varepsilon_e}\eta$, where the exponents $\varepsilon_i$ are defined by
$$
\varepsilon_i =\left\{\begin{array}{ccc}
1 & \text{if} & x_i=0 \text{ is invariant}.\\
0 & \text{if} & x_i=0 \text{ is dicritical}.
\end{array}\right.
$$
A local generator $\eta$ of $\mathcal{L}_{\mathcal{F}}$ is also called a \emph{local generator of $\mathcal{F}$ adapted to $E$}.
\begin{defin}
	An \emph{ambient surface} is a pair  $\mathcal{M}=(M,E)$, where $M$ is a nonsingular complex analytic surface $M$ and $E$ is a strong normal crossings divisor. A \emph{foliated surface} $(\mathcal{M},{\mathcal F})$ is the data of an ambient surface and a foliation $\mathcal F$ on $M$.
\end{defin}

Given an open subset $U\subset |M|$ such that $E \cap U$ has only finitely many irreducible components, the restriction $\mathcal{M}|_U$ is a well-defined ambient surface. In this case, the restriction $(\mathcal{M},{\mathcal F})|_U$ is also a foliated surface. Given a point $p\in |M|$, we define the germ $(\mathcal{M},{\mathcal F})_p$ of $(\mathcal{M},{\mathcal F})$ at $p$ in the natural way.
\subsection{Presimple and simple points under blowing-ups} \label{sec:presimples}
Let us consider a foliated surface $(\mathcal{M},\mathcal{F})$ and a point $p\in |M|$. A germ $\xi$ of holomorphic vector field is \emph{tangent to $\mathcal{F}$} if $\omega(\xi)=0$, where $\omega$ is a local generator of $\mathcal{F}$ (the sheaf of tangent germs of vector fields also defines the foliation). Notice that $\xi(p)=0$ when $p \in \text{Sing}(\mathcal{F})$ and, in this case, we have a well-defined linear part $L_{\xi}$.
\begin{defin}
	We say that $p$ is a \emph{presimple point for $(\mathcal{M},\mathcal{F})$} if $p \notin \text{Sing}(\mathcal{F},E)$ or we have that $p \in \text{Sing}(\mathcal{F})$, $e_p(E)\geq 1$, $e_p(E_{\text{dic}})=0$ and there is a germ $\xi$ of vector field tangent to $\mathcal{F}$ such that $L_{\xi}$ is non-nilpotent. We say that $p$ is \emph{simple} if it is presimple and the eigenvalues of $L_{\xi}$ have not positive rational ratio, when $p \in \text{Sing}(\mathcal{F})$. A \emph{saddle-node} is a simple singularity where the linear part of every tangent germ of vector field has a zero eigenvalue.
\end{defin}
\begin{rem}
	For a simple singularity that is not a saddle-node, the ratios $\alpha$ and $1/\alpha$ of the eigenvalues correspond to the Camacho-Sad indices with respect to the invariant branches (see \cite{Cam-S}).
\end{rem}
We distinguish two types of presimple points: trace and corner type points. More precisely:
\begin{enumerate}
	\item If $p \notin \text{Sing}(\mathcal{F})$, we say that it is of \emph{trace type} if $e_p(E_{\text{inv}})=0$ and that it is of \emph{corner type} if $e_p(E_{\text{inv}})=1$.
	\item If $p \in \text{Sing}(\mathcal{F})$, we say that it is of \emph{trace type} if $e_p(E)=1$ and that it is of \emph{corner type} if $e_p(E)=2$.
\end{enumerate}
\begin{rem}
	Given a presimple singularity, there are no dicritical components through it. If $p$ is a simple nonsingular point, we have that $e_p(E_\text{dic}) \in \{0,1\}$. Hence, a point $p$ with $e_p(E_\text{dic})=2$ cannot be a presimple point.
\end{rem}
\begin{defin}
	A foliated surface $(\mathcal{M},\mathcal{F})$ is \emph{desingularized} (respectively, \emph{pre-desingularized}) if it has only simple points (respectively, presimple points).
\end{defin}
\begin{rem} \label{rem:singsingadapt}
	If $(\mathcal{M},\mathcal{F})$ is pre-desingularized, then $\text{Sing}(\mathcal{F},E)=\text{Sing}(\mathcal{F})$.
\end{rem}
Consider the blowing-up $\pi: M' \rightarrow M$ centered at a point $p\in |M|$. We obtain a new foliated surface $(\mathcal{M}', \mathcal{F}')$, where $\mathcal{M}'=(M',E')$, with $E'=\pi^{-1}(E\cup \{p\})$ and $\mathcal{F}'$ is the transform of $\mathcal{F}$ by $\pi$. We write for short $\pi:(\mathcal{M}', \mathcal{F}') \rightarrow (\mathcal{M}, \mathcal{F})$.

We summarize now the main properties of the behaviour of simple and presimple points under blowing-up (for more details, see \cite{Can-C-D}.)
Let us assume that the center $p$ of $\pi$ is a presimple point for $(\mathcal{M}, \mathcal{F})$. We have that all $q\in D$ are presimple points for $(\mathcal{M}', \mathcal{F}')$, where $D=\pi^{-1}(p)$ is the exceptional divisor. More specifically:
\begin{enumerate}
	\item If $p \notin \text{Sing}(\mathcal{F})$, then $D$ is invariant and there is only one point $p'\in \text{Sing}(\mathcal{F}')\cap D$. Moreover $p'$ is a simple singularity (it represents the tangent at $p$ of the only invariant branch of $\mathcal{F}$ at $p$).
	\item If $p$ is a simple singularity for $(\mathcal{M}, \mathcal{F})$, then $D$ is invariant and there are exactly two points $p'_1,p'_2 \in \text{Sing}(\mathcal{F}')\cap D$. Moreover $p'_1$ and $p'_2$ are simple singularities for $(\mathcal{M}', \mathcal{F}')$. We have also that  $p'_1,p'_2$ are corners if $p$ is a corner and that there is a corner and a trace between $\{p'_1,p'_2\}$ if $p$ is a trace.
	\item If $p$ is a presimple but not simple singularity for $(\mathcal{M}, \mathcal{F})$, we take a germ of vector field $\xi$ with $L_{\xi}$ non-nilpotent. We have three possibilities:
	\begin{enumerate}
		\item The linear part $L_{\xi}$ is the identity up to a factor. In this case $D$ is dicritical and $\text{Sing}(\mathcal{F}')\cap D=\emptyset$.
		\item The linear part $L_{\xi}$ is not diagonalizable. The exceptional divisor $D$ is invariant and there is only one point $p'$ in $\text{Sing}(\mathcal{F}')\cap D$. Moreover $p'$ is a saddle-node corner type singularity. This situation only holds if $p$ is of trace type.
		\item The linear part $L_{\xi}$ has two different eigenvalues $\lambda,\mu$ with $\lambda/\mu \in \mathbb{Q}_{>1}$. In this situation $D$ is invariant and there are exactly two points $p',q' \in \text{Sing}(\mathcal{F}')\cap D$. One of them is a simple singularity and the other one is presimple with eigenvalues equal to $\lambda-\mu, \mu$. Moreover, if $p$ is of corner type, then $p',q'$ are both corners; if $p$ is of trace type, we obtain a corner and a trace between $\{p',q'\}$.
	\end{enumerate}
\end{enumerate}
\subsection{Reduction of singularities}
We say that a morphism $\pi:(\mathcal{M}',\mathcal{F}')\rightarrow(\mathcal{M},\mathcal{F})$ is a \emph{reduction {\rm (respectively}, a pre-reduction{\rm)} of singularities of $(\mathcal{M},\mathcal{F})$} if $\pi$ is a finite composition $\pi=\pi_1 \circ \pi_2 \circ \cdots \circ \pi_N$, where each
$$
\pi_i:(\mathcal{M}^i,\mathcal{F}^i)\rightarrow (\mathcal{M}^{i-1},\mathcal{F}^{i-1}),
$$
is a blowing-up centered at $p_{i-1}\in |M_{i-1}|$, for $i=1,2,\ldots,N$ and $(\mathcal{M}',\mathcal{F}')=(\mathcal{M}^N,\mathcal{F}^N)$ is a desingularized (respectively, pre-desingularized) foliated surface. The reduction of singularities $\pi$ is called \emph{minimal} if, for any other reduction of singularities $\bar{\pi}:(\mathcal{M}'',\mathcal{F}'')\rightarrow(\mathcal{M},\mathcal{F})$, there is a unique factorization $\bar{\pi} = \pi \circ f$, where
$$
f:(\mathcal{M}'',\mathcal{F}'')\rightarrow (\mathcal{M}',\mathcal{F}')
$$
is the composition of a finite sequence of blowing-ups (up to isomorphism). Note that $\pi$ is minimal if and only if all the centers $p_i$ are non-simple points. In the same way, we define and characterize \emph{minimal pre-reduction of singularities}.
\begin{rem}
	Following the above definitions, there is no reduction of singularities for foliated surfaces with infinitely many non-simple points.
\end{rem}
The following result is consequence of Seidenberg's Theorem \cite{Sei} and the 
statements in \cite{Can-C-D}.
\begin{teo}
	Let $(\mathcal{M},\mathcal{F})$ a foliated surface. We have that:
	\begin{enumerate}
		\item There is a reduction of singularities of $(\mathcal{M},\mathcal{F})$ if and only if the set of non-simple points is finite. In this case, there is a minimal reduction of singularities.
		\item There is a pre-reduction of singularities of $(\mathcal{M},\mathcal{F})$ if and only if the set of  non-presimple points is finite. In this case, there is a minimal pre-reduction of singularities. \qedhere
	\end{enumerate}
\end{teo}
\begin{rem}
	Assume that $(\mathcal{M},\mathcal{F})$ has reduction of singularities and let $\pi$ and $\sigma$ be, respectively, the minimal reduction and pre-reduction of singularities. If
	$$
	\pi:(\mathcal{M}',\mathcal{F}')\rightarrow(\mathcal{M},\mathcal{F}), \quad \sigma:(\mathcal{M}^*,\mathcal{F}^*)\rightarrow(\mathcal{M},\mathcal{F}),
	$$
	then there is a unique factorization $\pi = \sigma \circ f$, where $f:(\mathcal{M}',\mathcal{F}')\rightarrow (\mathcal{M}^*,\mathcal{F}^*)$
	is the composition of a finite sequence of blowing-ups centered at presimple but non-simple points.
\end{rem}
We are interested in foliated surfaces without saddle-nodes after reduction of singularities. In other contexts they correspond to the so-called ``generalized curves'' (see \cite{Cam-N-S}).

A foliated surface $(\mathcal{M},\mathcal{F})$ is \emph{complex hyperbolic} (for short, CH) if there is a reduction of singularities $\pi:(\mathcal{M}',\mathcal{F}')\rightarrow(\mathcal{M},\mathcal{F})$ without saddle-nodes in $(\mathcal{M}',\mathcal{F}')$. Note the following properties:
\begin{enumerate}
	\item If $(\mathcal{M},\mathcal{F})$ is CH, for any reduction of singularities $\bar{\pi}:(\mathcal{M}'',\mathcal{F}'')\rightarrow(\mathcal{M},\mathcal{F})$ there are no saddle-nodes in $(\mathcal{M}'',\mathcal{F}'')$.
	\item Being CH depends just on the foliation $\mathcal{F}$. That is, if $((M,E),\mathcal{F})$ is CH, then any other $((M,D),\mathcal{F})$ is also CH. In particular, it makes sense to say that a point $p\in |M|$ is CH for $\mathcal{F}$ when the germ $\mathcal{F}_p$ is CH, independently of the chosen divisor.
	\item A simple singularity is CH if and only if it is not a saddle-node. The presimple but not simple singularities that are not CH are those presented in 3b) of Subsection \ref{sec:presimples}.
\end{enumerate}
\begin{lema} \label{lema:presimplesCH}
	A point $p$ of a CH-foliated surface $(\mathcal{M},\mathcal{F})$ is presimple if and only if $\nu_p(\mathcal{F},E)=0$.
\end{lema}
\begin{proof}
	If $e=e_p(E)=0$, then $p$ is presimple if and only if $p$ is nonsingular, and we are done. If $e=1$, take a local generator $\eta=a_1dx_1/x_1+a_2dx_2$ of $\mathcal{F}$ adapted to $E$. We distinguish two cases:
	
	\begin{enumerate}
		\item [1a)] Dicritical case, that is $E=(x_1=0)$ is dicritical. This is equivalent to say that $a_1=x_1\tilde{a_1}$, in particular, we have $\nu_p(a_1)>0$. A local generator of $\mathcal{F}$ is given by $\omega=\tilde{a_1}dx_1+a_2dx_2$. In this case, the point $p$ is presimple if and only if it is nonsingular and $\mathcal{F}, E$ have normal crossings at $p$. That happens if and only if $\nu_p(a_2)=0$.	
		\item [1b)] Invariant case. A local generator of $\mathcal{F}$ is given by $\omega=a_1dx_1+x_1a_2dx_2$. The point $p$ is nonsingular with normal crossings if and only if $\nu_p(a_1)=0$. The point $p$ is a presimple singularity if and only if $p$ is a saddle-node (not allowed, because we are in the CH-situation) or $\nu_p(a_2)=0$.
	\end{enumerate}
	If $e=2$, consider a local generator $\eta=a_1dx_1/x_1+a_2dx_2/x_2$ of $\mathcal{F}$ adapted to $E$, where $E=(x_1x_2=0)$ locally at $p$. There are three possibilities:
	\begin{enumerate}
		\item [2a)] $e_p(E_{\text{dic}})=2$. We know that $p$ is not presimple. On the other hand, we have that $x_i$ divides $a_i$ for $i=1,2$, hence $\nu_p(a_1,a_2)>0$.
		\item [2b)] $e_p(E_{\text{dic}})=1$. Suppose that $(x_1=0)$ is invariant and that $(x_2=0)$ is dicritical. We have that $a_2=x_2\tilde{a}_2$ and that $\omega=a_1dx_1+x_1\tilde{a}_2dx_2$ is a local generator of $\mathcal{F}$. The point $p$ is presimple if and only if it is a nonsingular point and $\mathcal{F}, E$ have normal crossings at $p$. This happens if and only if $\nu_p(a_1)=0$. We end by noting that $\nu_p(a_1)=0$ if and only if $\nu_p(a_1,a_2)=0$, since $\nu_p(a_2)>0$.
		\item [2c)] $e_p(E_{\text{dic}})=0$. We have that $\omega=x_2a_1dx_1+x_1a_2dx_2$ is a local generator of $\mathcal{F}$ and hence $p$ is a singular point. The linear part of a germ of vector field tangent to $\mathcal{F}$ is diagonal with eigenvalues $\lambda_i=a_i(0)$, for $i=1,2$. As a consequence, the point $p$ is presimple if and only if $\nu_p(a_1,a_2)=0$. \qedhere
	\end{enumerate}
\end{proof}
\subsection{Combinatorial blowing-ups}
The concept of toric type foliated surface has been introduced in \cite{Cam-C}.

Let $(\mathcal{M},\mathcal{F})$ be a foliated surface and let us consider the blowing-up $\pi:(\mathcal{M}', \mathcal{F}') \rightarrow (\mathcal{M}, \mathcal{F})$ centered at a point $p\in |M|$. We say that $\pi$ is a \emph{combinatorial blowing-up} if $e_p(E)=2$. The composition of a finite sequence of blowing-ups is \emph{combinatorial} if each blowing-up is combinatorial.
\begin{defin}
	A foliated surface $(\mathcal{M},\mathcal{F})$ is of \emph{toric type} if it admits a combinatorial reduction of singularities. Analogously, we say that $(\mathcal{M},\mathcal{F})$ is of \emph{weak toric type} if it has a combinatorial pre-reduction of singularities.
\end{defin}
The foliated surface $(\mathcal{M},\mathcal{F})$ is called of \emph{toric type at a point $p\in |M|$} if the germ $(\mathcal{M},\mathcal{F})_p$ is of toric type. Analogously, it is called of \emph{weak toric type at $p$} if the germ $(\mathcal{M},\mathcal{F})_p$ is of weak toric type.
\begin{rem} \label{rem:tipotoricoenpuntosnoesquina}
	Given a point $p$ with $e_p(E)\leq 1$, we have that $(\mathcal{M},\mathcal{F})$ is of toric type (respectively of weak toric type) at $p$ if and only if $p$ is a simple (respectively presimple) point for $(\mathcal{M},\mathcal{F})$.
\end{rem}
Recall that the divisor $E$ has finitely many irreducible components. Thus, there are only finitely many points with $e_p(E)=2$. As a consequence, the foliated surface $(\mathcal{M},\mathcal{F})$ is of toric type if and only if the property holds at every point $p\in |M|$. We have the same comment for the weak toric type property.

\begin{rem} \label{rem:presimpletipotorico}
	If $p$ is a corner type presimple singularity of $(\mathcal{M},\mathcal{F})$, we have that $(\mathcal{M},\mathcal{F})$ is of toric type at $p$, in view of the behaviour of presimple singularities described in Subsection \ref{sec:presimples}. Even more, all the singularities appearing after the minimal reduction of singularities of the germ $(\mathcal{M},\mathcal{F})_p$ are corner type simple points.
\end{rem}

\subsection{Isolated invariant branches}
The concept of isolated invariant branch is useful for choosing finitely many representative invariant branches, in case dicritical components arise after reduction of singularities. We give the precise definition below.

\begin{defin}
	Consider an invariant branch $(\Gamma,p)$ of a foliated surface $(\mathcal{M},\mathcal{F})$. We say that $(\Gamma,p)$ is \emph{isolated for $(\mathcal{M},\mathcal{F})$} if the following properties hold:
	\begin{itemize}
		\item $(\Gamma,p)\not\subset (E,p)$.
		\item For every composition $\pi:(\mathcal{M}',\mathcal{F}')\rightarrow (\mathcal{M},\mathcal{F})$ of a finite sequence of blowing-ups, we have that $p'\in \text{Sing}(\mathcal{F}',E')$, where $(\Gamma',p')$ is the strict transform of $(\Gamma,p)$ by $\pi$. \qedhere
	\end{itemize}
\end{defin}
\begin{rem}
	Note that $p \in \text{Sing}(\mathcal{F},E)$, just by taking $\pi$ the identity. Besides, it is enough to consider blowing-ups centered at the infinitely near points of $(\Gamma,p)$.
\end{rem}
The property of being isolated is stable under blowing-ups. In the following statement we show that it is also stable by combinatorial blowing-downs.
\begin{prop}  \label{prop:estabilidadaisladas}
	Let us consider a combinatorial blowing-up $\pi:(\mathcal{M}',\mathcal{F}')\rightarrow (\mathcal{M}, \mathcal{F})$ between foliated surfaces. If $(\Gamma',p')$ is an isolated invariant branch for $(\mathcal{M}',\mathcal{F}')$, we have that $(\Gamma,p)$ is an isolated invariant branch for $(\mathcal{M}, \mathcal{F})$, where $(\Gamma,p)$ is the image of $(\Gamma',p')$ by $\pi$.
\end{prop}
\begin{proof}
	Suppose that $p$ is the center of the blowing-up, otherwise we are done. If $(\Gamma',p')$ is isolated, in particular $(\Gamma',p') \not\subset (E',p')$ and then also $(\Gamma,p)\not\subset (E,p)$. If $p \in \text{Sing}(\mathcal{F},E)$, we are done. Assume that $p \notin \text{Sing}(\mathcal{F},E)$. Since $p$ is a regular point, we get that $(\Gamma,p)$ is the only invariant branch through $p$. Moreover, we have that $e_p(E)=2$ and that $\mathcal{F}, E$ have normal crossings at $p$. Then $(\Gamma,p)$ is contained in $E$ and this is a contradiction.
\end{proof}
The following assertions give a description of the isolated invariant branches at presimple points.
\begin{lema} \label{lema:ramasaisladaspuntostraza}
	Consider a foliated surface {\rm $(\mathcal{M},\mathcal{F})$} and $\pi: (\mathcal{M}',\mathcal{F}') \rightarrow (\mathcal{M},\mathcal{F})$ a reduction of singularities. Let $(\Gamma,p)$ be an isolated branch of $(\mathcal{M},\mathcal{F})$ and let $(\Gamma',p')$ be the strict transform of $(\Gamma,p)$ by $\pi$. We have that $p'$ is a trace type simple singularity.
\end{lema}	
\begin{proof}
	By definition, we know that $p'\in \text{Sing}(\mathcal{F}',E')$ and $(\Gamma',p') \not\subset (E',p')$. Moreover, we have
	$$
	\text{Sing}(\mathcal{F}',E')=\text{Sing}(\mathcal{F}'),
	$$
	because $(\mathcal{M}',\mathcal{F}')$ is desingularized. If $p'$ is of corner type, the only invariant branches through it are contained in $E'$. Then $p'$ is a trace type simple singularity.
\end{proof}
\begin{lema} \label{lema:aisladaspresimple}
	Let $p\in \text{Sing}(\mathcal{F})$ be a presimple singularity of a foliated surface {\rm $(\mathcal{M},\mathcal{F})$}. If $p$ is of corner type, there are no isolated invariant branches through it. If $p$ is of trace type, there is at most one isolated invariant branch through it; when it exists, it is nonsingular, transversal to the divisor and any other nonsingular invariant branch is tangent to the divisor.
\end{lema}
\begin{proof}
	Assume first that $p$ is of corner type and let us find a contradiction with the existence of an isolated invariant branch $(\Gamma,p)$. By Remark \ref{rem:presimpletipotorico}, we know that the minimal reduction of singularities of the germ $(\mathcal{M},\mathcal{F})_p$ produces only singularities that are of corner type. In view of Lemma \ref{lema:ramasaisladaspuntostraza}, the strict transform of $(\Gamma,p)$ passes through a trace type simple singularity and this is not possible.
	
	Assume now that $p$ is a trace type presimple singularity. Recall that $e_p(E)=e_p(E_{\text{inv}})=1$. Suppose that there is an isolated branch $(\Gamma,p)$. Consider the blowing-up $\pi:(\mathcal{M}',\mathcal{F}')\rightarrow (\mathcal{M},\mathcal{F})$ centered at $p$. Let $(\Gamma',p')$ be the strict transform of $(\Gamma,p)$. We know that $p' \in \text{Sing}(\mathcal{F}',E')=\text{Sing}(\mathcal{F}')$. Since $(\Gamma',p')$ is an isolated invariant branch, we obtain that $p'$ is not a corner type point, hence it is a trace type presimple singularity of $(\mathcal{M}',\mathcal{F}')$. In particular, we have that $e_{p'}(E')=e_{p'}(E'_{\text{inv}})=1$. Note also that $E'=\pi^{-1}(p)$, locally at $p'$.
	
	The above arguments show that $(\Gamma,p)$ is nonsingular and transversal to $E$. Indeed, this is a direct consequence of the fact that the infinitely near points of $(\Gamma,p)$  are never over the strict transform of the precedent exceptional divisor.
	
	Let us prove that any other nonsingular invariant branch is tangent to $E$. Equivalently, if $(\Gamma_1,p)$ is a nonsingular invariant branch transversal to $E$, let us show that $(\Gamma_1,p)=(\Gamma,p)$. Denote by $(\Gamma_1',p_1')$ to the strict transform of $(\Gamma_1,p)$ by $\pi$. It is enough to prove that $p_1'=p'$; in this case, the situation repeats at $p'$,  we conclude that $(\Gamma_1,p)$ and $(\Gamma,p)$ have the same infinitely near points and thus they coincide. Since $(\Gamma_1,p)$ is transversal to $E$, we have that $p_1'$ does not belong to the strict transform of $E$. Moreover, $p_1'\in \text{Sing}(\mathcal{F'})$, because both $(\Gamma'_1,p'_1)$ and $\pi^{-1}(p)$ define invariant branches of $\mathcal{F'}$ at $p'_1$. Then, the only trace type singularity of $(\mathcal{M}',\mathcal{F}')$ in $\pi^{-1}(p)$ is $p_1'$; this means that $p_1'=p'$.
	
	It remains to show that $(\Gamma,p)$ is the only isolated invariant branch. We know that any isolated invariant branch must be nonsingular and transversal to $E$, then it is necessarily $(\Gamma,p)$.
\end{proof}	
\begin{cor}
	Let $(\Gamma,p)$ be a smooth invariant branch transversal to $E$ through a trace type presimple singularity. If there is an isolated branch for $(\mathcal{M},\mathcal{F})$ through $p$, it is necessarily $(\Gamma,p)$.
\end{cor}

\begin{rem} \label{rem:simpleaislada}
	If $p$ is a trace type non saddle-node simple singularity, there is exactly one isolated invariant branch through $p$. When it is a saddle-node, it is possible to have a formal non-convergent invariant branch, that is also isolated (in the formal sense).
\end{rem}
\begin{rem}
	Consider a foliated surface $(\mathcal{M}, \mathcal{F})$ and a point $p\in \text{Sing}(\mathcal{F},E)$. If there are only finitely many invariant branches through $p$ (equivalently, there are no dicritical components created after reduction of singularities over $p$), then each invariant branch $(\Gamma,p) \not\subset (E,p)$ is isolated.
\end{rem}

\section{Pairs of Laurent Polynomials in General Position}
We recall a result about the number of solution of a Laurent polynomials system looking at the mixed area of the convex polytopes associated to it (see \cite{Ber, Kus}). In Section \ref{sec:Newton Non-degenerate Foliations on Projective Toric Surfaces}, we apply it in our study of Newton non-degenerate foliations in projective toric ambient surfaces.

A \emph{convex polytope} $\Delta \subset \mathbb{R}^2$ is the convex hull of a finite set of points of $\mathbb{R}^2$. We denote the area of $\Delta$ by $\text{Ar}(\Delta)\in \mathbb{R}_{\geq 0}$. The \emph{mixed area} MA$(\Delta_1,\Delta_2)$ of two convex polytopes $\Delta_1,\Delta_2 \subset \mathbb{R}^2$ is given by $\text{MA}(\Delta_1,\Delta_2)=\text{Ar}(\Delta_1 + \Delta_2)-\text{Ar}(\Delta_1)-\text{Ar}(\Delta_2)$,
where $\Delta_1 + \Delta_2$ denotes the Minkowski sum.
\begin{rem} \label{rem:areamixta0}
	The mixed area of two convex polytopes is zero if and only if either one of them is a single point or they are parallel segments.
\end{rem}
The ring of Laurent polynomials $\mathbb{C}[u_1,u_2,u_1^{-1},u_2^{-1}]$ gives the regular functions of the complex torus $(\mathbb{C}^*)^2=(\mathbb{C}\setminus \{0\})^2$. The \emph{convex polytope $\Delta(f)$ of a Laurent polynomial $f$} is defined by
$$
\Delta(f)=\text{convex hull}\Big(\{(i,j)\in \mathbb{Z}^2; \; f_{ij}\neq 0\}\Big), \quad f=\sum\nolimits_{(i,j)}f_{ij}u_1^iu_2^j.
$$
The \emph{support-restriction $f_C$ of $f$ to a subset $C\subset \mathbb{Z}^2$} is $f_C=\sum\nolimits_{(i,j)\in C}f_{ij}u_1^iu_2^j$.

We call \emph{weight vectors} to the elements of the set $\mathcal{W}=\{(p,q)\in \mathbb{Z} \times \mathbb{Z}_{>0}; \; p\mathbb{Z}+q\mathbb{Z}=\mathbb{Z}\} \cup \{(1,0)\}$.
Note that, there is a bijection $\mathcal{W}\rightarrow \mathbb{Q} \cup \{\infty\}$, given by $(p,q)\mapsto -p/q$ (assuming $-1/0=\infty$). We say that a Laurent polynomial $F\ne 0$ is \emph{quasi-homogeneous with weight vector $(p,q)$} if there is an $r \in \mathbb{Z}$ such that $F(t^pu_1,t^qu_2)=t^{r}F(u_1,u_2)$. The integer $r$ is called the \emph{quasi-homogeneous degree of $F$}. In this case, there is a decomposition
$$ \label{eq:casi-homogeneopositivo}
F=cu_1^{\tau_1}u_2^{\tau_2}\prod\nolimits_{j=1}^N(u_2^p-\alpha_ju_1^q); \quad \; (\tau_1,\tau_2) \in \mathbb{Z}^2, \;c, \alpha_j \in \mathbb{C}^*,\, j=1,2,\ldots N.
$$
Conversely, such a decomposition provides a quasi-homogeneous Laurent polynomial. Moreover, we can see that $F$ is quasi-homogeneous with weight vector $(p,q)$ and degree  $r$ if and only if $\Delta(F)$ is a segment contained in the line of equation $pi+qj=r$. In particular, a single monomial is quasi-homogeneous for any weight vector.

The following definitions can be found in \cite{Ber, Kus}.
\begin{defin} \label{def:nodegenerado}
	A pair $(F_1,F_2)$ of quasi-homogeneous Laurent polynomials with weight vector $(p,q)$ is \emph{non-degenerate} if $\alpha_{1j}\neq \alpha_{2k}$ for every $j,k$, where
	$$
	F_i=c_iu_1^{\tau_{i1}}u_2^{\tau_{i2}}\prod\nolimits_{j=1}^{N_i}(u_2^p-\alpha_{ij}u_1^q), \quad i=1,2.
	$$
	Otherwise, we say that $(F_1,F_2)$ is \emph{degenerate}.
\end{defin}
\begin{defin} \label{def:posiciongeneral}
	A pair of arbitrary Laurent polynomials $(f_1,f_2)$ is in \emph{general position}, if the pairs $(f_{1L}, f_{2L})$ are non-degenerate for every side $L$ of $\Delta(f_1,f_2)=\text{convex hull}(\Delta(f_1)\cup \Delta(f_2))$.
\end{defin}
\begin{rem} \label{rem:cambiomonomial}
	Consider two Laurent polynomials $F(u_1,u_2),G(u_1,u_2)$ and three monomials $u^a=u_1^{a_1}u_2^{a_2}$, $u^b=u_1^{b_1}u_2^{b_2}$,  $u^c=u_1^{c_1}u_2^{c_2}$, where $a,b,c \in \mathbb{Z}^2$ and $a_1b_2-a_2b_1\in\{-1,1\}$. We have the following properties:
	\begin{enumerate}
		\item The Laurent polynomial $F$ is quasi-homogeneous if and only if $u^cF$ is quasi-homogeneous. Moreover, both have the same weight vector.
		\item The Laurent polynomial $F(u_1,u_2)$ is quasi-homogeneous with weight vector $(p,q)$ if and only $F(u^a,u^b)$ is quasi-homogeneous with weight vector (up to sign) $(pb_2-qa_2,qa_1-pb_1)$.
		\item The pair $(F(u_1,u_2),G(u_1,u_2))$ is non-degenerate if and only if the pair $(u^cF(u^a,u^b),u^cG(u^a,u^b))$ is non-degenerate.\qedhere		
	\end{enumerate}
\end{rem}
\begin{rem}\label{rem:suma}
	If the pair $(F_1,F_2)$ is non-degenerate and both $F_1,F_2$ have the same quasi-homogeneous degree, then $(F_1+F_2,F_2)$ is also non-degenerate.
\end{rem}
\begin{rem} \label{rem:una de las restricciones es cero}
	Note that $(f_{1L}, f_{2L})$ is degenerate if $f_{1L}=0$ or $f_{2L}=0$.
\end{rem}
\begin{teo} \label{teo:KKB}
	(See \cite{Ber, Kus}) Let $(f_1,f_1)$ be a pair of Laurent polynomials in general position. The number of solutions in $(\mathbb{C}^*)^2$ of the system $f_1=f_2=0$ is equal to the mixed area $\text{MA}(\Delta(f_1),\Delta(f_2))$.
\end{teo}
\section{Weak Toric Type Foliated Surfaces} \label{sec:Weak Toric Type Foliated Surfaces}
We give an algebraic characterization in terms of  ``weighted initial forms'' of the weak toric type complex hyperbolic foliated surfaces. More precisely, we introduce the concept of Newton non-degenerate foliated surface, following the classical ideas for curves and functions (see \cite{Oka}). We proof that they are exactly the weak toric type foliated surfaces in the CH-context.
\subsection{Newton non-degenerate foliated surfaces}
The \emph{Newton polygon $N(f;x_1,x_2)$} of a formal power series $f=\sum f_{ij}x_1^ix_2^j \in \mathbb{C}[[x_1,x_2]]$ is defined by
$$
N(f;x_1,x_2)=\text{convex hull}(\text{Supp}(f)+\mathbb{R}_{\geq 0}^2),
$$
where $\text{Supp}(f)=\{(i,j)\in \mathbb{Z}^2; \; f_{ij}\neq 0\}$. The topological boundary of $N(f;x_1,x_2)$ is a union of two non-compact sides and finitely many compact sides (consisting of more than one point) with negative rational slopes. The endpoints of the sides are called vertices.

Let us consider an ambient surface $\mathcal{M}=(M,E)$ and a point $p\in |M|$ with $e_p(E)=2$. Take a logarithmic one-form $\eta \in \Omega^1_{M,p}(\log E)$, that we write in adapted local coordinates $(x_1,x_2)$ as
$$
\eta=a_1dx_1/x_1+a_2dx_2/x_2; \quad a_1,a_2\in \mathcal{O}_{M,p} = \mathbb{C}\{x_1,x_2\} \subset \mathbb{C}[[x_1,x_2]].
$$
The \emph{Newton polygon} $N_p(\eta; x_1,x_2)$ is the convex hull of $N(a_1;x_1,x_2)\cup N(a_2;x_1,x_2)$.
\begin{rem}
	In order to get uniqueness in the definition of the Newton polygon of a foliated surface, we consider total orderings $\prec$ in the set of irreducible components of the divisor.
\end{rem}
Now, we consider a foliated surface $(\mathcal{M},\mathcal{F})$, a total ordering $\prec$ in the set of irreducible components of $E$ and a point $p\in |M|$ with $e_p(E)=2$.
\begin{lema} \label{lema:orden}
	Take local coordinates $(x_1,x_2)$, $(x_1',x_2')$ adapted to $E$, such that $(x_2=0) \prec (x_1=0)$ and $(x'_2=0) \prec (x'_1=0)$. We have that
	$$
	N= N_p(\eta; x_1,x_2)=N_p(\eta'; x'_1,x'_2),
	$$
	where $\eta=a_1dx_1/x_1+a_2dx_2/x_2$, $\eta'=a'_1dx'_1/x'_1+a'_2dx'_2/x'_2$ are local generators of $\mathcal{F}$ adapted to $E$. Moreover, given a compact side $L$ of $N$, it holds that $(a_{1L},a_{2L})$ is non-degenerate if and only if $(a'_{1L},a'_{2L})$ is non-degenerate.
\end{lema}
\begin{proof}
	Just note that there are units $u,u_1,u_2\in \mathcal{O}_{M,p}$, such that $\eta'=u\eta$ and $x_i'=u_ix_i$, for $i=1,2$.	
\end{proof}
In view of Lemma \ref{lema:orden} we define the \emph{Newton polygon $N^{\prec}_p(\mathcal{M},\mathcal{F})$} by
$$
N^{\prec}_p(\mathcal{M},\mathcal{F})=N_p(\eta; x_1,x_2).
$$
\begin{rem} \label{rem:presimpleunsolovertice}
	The point $p$ is presimple if and only if $N^{\prec}_p(\mathcal{M},\mathcal{F})=\mathbb{R}_{\geq 0}^2$.
\end{rem}
The foliated surface $(\mathcal{M},\mathcal{F})$ is \emph{non-degenerate at $p$ with respect to the order $\prec$ and a compact side $L$ of $N^{\prec}_p(\mathcal{M},\mathcal{F})$} if the pair $(a_{1L},a_{2L})$ is non-degenerate. For short, we say that $(\mathcal{F},L)$ is \emph{non-degenerate}.

\begin{rem} \label{rem:relacionordenes}
	Consider two different orderings $\prec$, $\prec'$ in the set of irreducible components of $E$. Denote by $E_1$ and $E_2$ the irreducible components of $E$ through $p$, such that $E_2 \prec E_1$. We have two possibilities at $p$ depending on the order $\prec'$ between $E_1$ and $E_2$:
	\begin{itemize}
		\item $E_2 \prec' E_1$. We have $N^{\prec'}_p(\mathcal{M},\mathcal{F})=N^{\prec}_p(\mathcal{M},\mathcal{F})$.
		\item $E_1 \prec' E_2$. We have $N^{\prec'}_p(\mathcal{M},\mathcal{F})=\sigma(N^{\prec}_p(\mathcal{M},\mathcal{F}))$, where $\sigma$ is the symmetry $(u,v)\mapsto (v,u)$. Moreover, given a compact side $L$ of $N^{\prec}_p(\mathcal{M},\mathcal{F})$, we have that $(\mathcal{F},L)$ is non-degenerate, with respect to $\prec$, if and only if $(\mathcal{F},\sigma(L))$ is non-degenerate, with respect to $\prec'$. \qedhere
	\end{itemize}
\end{rem}
\begin{defin}
	A foliated surface $(\mathcal{M},\mathcal{F})$ is \emph{Newton non-degenerate at $p\in |M|$} if the point $p$ is presimple, or $e_p(E)= 2$ and $(\mathcal{F},L)$ is non-degenerate for each compact side $L$ of $N^{\prec}_p(\mathcal{M},\mathcal{F})$ (this definition does not depend on the chosen ordering $\prec$, in view of Remark \ref{rem:relacionordenes}). We say that $(\mathcal{M},\mathcal{F})$ is \emph{Newton non-degenerate} if the property holds at each point.
\end{defin}

\subsection{Non-degenerate foliations and combinatorial blowing-ups} \label{subsec:poligonoexplosiones}
We present here several results about the stability of being Newton non-degenerate under combinatorial blowing-ups and blowing-downs.

Let us consider a CH-foliated surface $(\mathcal{M},\mathcal{F})$, a point $p\in |M|$ with $e_p(E)=2$ and a total ordering $\prec$ in the set of irreducible components of $E$. Denote by $E_1$, $E_2$ the irreducible components of $E$ through $p$, such that $E_2 \prec E_1$. Let $(x_1,x_2)$ be a local coordinate system with $E_i=(x_i=0)$, for $i=1,2$. Consider a local generator $\eta=a_1dx_1/x_1+a_2dx_2/x_2$ of $\mathcal{F}$ adapted to $E$ at $p$. Denote by $d\geq 0$ the adapted order, that is $d=\nu_p(\mathcal{F},E)=\nu_p(a_1,a_2)$. We write
$$
a_i=A_i+\tilde{a}_i\in \mathbb{C}\{x_1,x_2\}, \; \nu_p(\tilde{a}_i) > d, \quad i=1,2,
$$
where the $A_i$ are homogeneous polynomials of degree $d$. Note that $(A_1,A_2)\ne (0,0)$.

Let us perform the blowing-up $\pi: (\mathcal{M}',\mathcal{F}') \rightarrow (\mathcal{M}, \mathcal{F})$ centered at $p$. Note that $\pi$ is a combinatorial blowing-up, since $e_p(E)=2$. We consider the ordering $\prec'$ in the set of irreducible components of $E'$ obtained by adding the exceptional divisor $D=\pi^{-1}(p)$ with the property $E'_2 \prec' D \prec' E'_1$, where $E'_i$ denote the strict transforms of $E_i$, for $i=1,2$. Denote
$$
\{q_0\}=D \cap E'_2, \; \{q_{\infty}\}=D \cap E'_1 \;  \text{ and } \mathcal{T}=D \setminus \{q_0,q_{\infty}\}.
$$
\begin{rem} \label{rem:coordenadas}
	We can consider affine coordinates $(x_1',x_2')$ on the chart of the blowing-up with origin in $q_0$, given by $x_1=x'_1,x_2=x'_1x'_2$. In this chart $E'_2=(x_2'=0)$ and $D=(x'_1=0)$. Analogously, we can consider affine coordinates $(x_1'',x_2'')$ on the chart of the blowing-up with origin in $q_{\infty}$, given by $x_1=x''_1x_2'',x_2=x''_2$. Here $D=(x''_2=0)$ and $E'_1=(x_1''=0)$. These choices of coordinates are compatible with the ordering  $\prec'$ in the set of irreducible components of $E'$.
\end{rem}
\begin{rem}\label{rem:ladopendiente-1}
	The following properties are well known (see \cite{Can-C-D}):
	\begin{enumerate}
		\item Assume that the blowing-up $\pi$ is dicritical. That is, the exceptional divisor $D$ is dicritical, what happens if and only if $A_1+A_2=0$. In this case $\mathcal{T} \cap \text{Sing}(\mathcal{F'},E')$ is given by $A_1=0$ and all the points in this set are  non-presimple points.
		\item Suppose that the blowing-up $\pi$ is non-dicritical. In this case	$\text{Sing}(\mathcal{F}',E') \cap D=\text{Sing}(\mathcal{F}') \cap D$ and $\mathcal{T} \cap \text{Sing}(\mathcal{F'})$ is given by the tangent cone $A_1+A_2=0$. \qedhere
	\end{enumerate}
\end{rem}
\begin{lema}\label{lema:ladopendiente-1}
	If holds that $\mathcal{T} \cap \text{Sing}(\mathcal{F'},E')\neq \emptyset$ if and only if $N^{\prec}_p(\mathcal{M},\mathcal{F})$ has a compact side of slope $-1$.
\end{lema}
\begin{proof}
	It follows by Remark \ref{rem:ladopendiente-1}, looking separately the dicritical and non-dicritical cases.
\end{proof}
\begin{lema} \label{lema:lugarnopresimple}
	If $\pi$ is non-dicritical, then the non-presimple points $q \in \mathcal{T}$ are given by $A_1+A_2=A_2=0$.
\end{lema}
\begin{proof}
	Consider coordinates $(x_1',x_2')$ like in Remark \ref{rem:coordenadas}. A generator $\eta'$ of $\mathcal{F}'$ adapted to $E'$ is given by
	$$
	\eta'=\left((A_1+A_2)(1,x_2')+x_1'\frac{(\tilde{a}_1+\tilde{a}_2)(x_1',x_1'x_2')}{(x_1')^{d+1}}\right)\frac{dx_1'}{x_1'}+
	\left(A_2(1,x_2')+x_1'\frac{\tilde{a}_2(x_1',x_1'x_2')}{(x_1')^{d+1}}\right)\frac{dx_2'}{x_2'},
	$$
	The  non-presimple points $q \in \mathcal{T}$ are given by $x'_1=0$, $x'_2=\lambda\in \mathbb{C}^*$, with  $(A_1+A_2)(1,\lambda)=A_2(1,\lambda)=0$, in view of Lemma \ref{lema:presimplesCH}.
\end{proof}
\begin{cor} \label{cor:ladopendiente-1}
	The pair $(A_1,A_2)$ is non-degenerate if and only if each point $q\in\mathcal{T}$ is presimple.
\end{cor}
\begin{proof}
	Suppose that $N^{\prec}_p(\mathcal{M},\mathcal{F})$ does not have a compact side with slope $-1$. This is equivalent to say that $A_1$ and $A_2$ are the same monomial up to constant and hence $(A_1,A_2)$ is a non-degenerate pair. On the other hand, in view of Lemma \ref{lema:ladopendiente-1}, all the points $q\in\mathcal{T}$ are presimple.
	
	Suppose now that $N^{\prec}_p(\mathcal{M},\mathcal{F})$ has a compact side with slope $-1$. Assume first that all the points $q\in\mathcal{T}$ are presimple. The blowing-up $\pi$ is non-dicritical by Remark \ref{rem:ladopendiente-1}, then Lemma \ref{lema:lugarnopresimple} implies that $(A_1+A_2,A_2)$ is non-degenerate and by Remark \ref{rem:suma}, this is equivalent to say that $(A_1,A_2)$ is non-degenerate. Conversely, if $(A_1,A_2)$ is non-degenerate, then $A_1 \ne -A_2$ and the blowing-up is non-dicritical. Moreover $(A_1+A_2,A_2)$ is non-degenerate, as a consequence, all the points $q\in\mathcal{T}$ are presimple.
\end{proof}
We introduce the notation $N=N^{\prec}_p(\mathcal{M},\mathcal{F}), \; N_0=N^{\prec'}_{q_0}(\mathcal{M}',\mathcal{F}'), \;N_{\infty}=N^{\prec'}_{q_{\infty}}(\mathcal{M}',\mathcal{F}').$

Let $\mathcal{S}_{p}^+$ be the set of compact sides of $N$ with slope greater than $-1$ and let $\mathcal{S}_{q_0}$ be the set of compact sides of $N_0$. There is a bijection between $\mathcal{S}_{p}^+$ and $\mathcal{S}_{q_0}$ given by $L \in \mathcal{S}_{p}^+ \mapsto L' \in \mathcal{S}_{q_0}$,
where $L'$ is the side of slope $m'=m/(1+m) <0$, being $m\in \mathbb{Q}_{>-1}$ the slope of $L$. In the same way, denote by $\mathcal{S}_{p}^-$ the set of compact sides of $N$ with slope lower than $-1$ and by $\mathcal{S}_{q_{\infty}}$ the set of compact sides of $N_{\infty}$. There is a bijection $L \in \mathcal{S}_{p}^- \mapsto L' \in \mathcal{S}_{q_{\infty}}$ where $L'$ has slope $m'=m+1 < 0$, being $m\in \mathbb{Q}_{<-1}$ the slope of $L$. Given a compact side $L$ of $N$ with slope different than $-1$, that is $L\in \mathcal{S}_{p}^+ \cup \mathcal{S}_{p}^-$, we say that $L'$ is the \emph{transform of $L$}.

\begin{lema} \label{lema:NNDq0}
	Consider a compact side $L\in \mathcal{S}_{p}^+ \cup \mathcal{S}_{p}^-$ and let $L'$ be transform of $L$. It holds that $(\mathcal{F},L)$ is non-degenerate if and only if $(\mathcal{F}',L')$ is non-degenerate.
\end{lema}
\begin{proof}
	Suppose that $L\in \mathcal{S}_{p}^+$, the proof is analogous when $L\in \mathcal{S}_{p}^-$. Consider coordinates $(x_1',x_2')$ like in Remark \ref{rem:coordenadas}. A local generator $\eta'$ of $\mathcal{F}'$ adapted to $E'$ at $q_0$ is given by
	$
	\eta'=a_1'dx_1'/x_1'+a_2'dx_2'/x_2',
	$
	where $a_1'=(a_1+a_2)(x_1',x_1'x_2')/(x_1')^d$ and $a'_2=a_2(x_1',x_1'x_2')/(x_1')^d$. We also have that
	$$
	a_{1L'}'=(a_{1L}+a_{2L})(x_1',x_1'x_2')/(x_1')^d, \;a'_{2L'}=a_{2L}(x_1',x_1'x_2')/(x_1')^d.
	$$
	Then $(a_{1L'}',a_{2L'}')$ is non-degenerate if and only if $(a_{1L}+a_{2L},a_{2L})$ is non-degenerate, in view of Remark \ref{rem:cambiomonomial}. This is equivalent to have $(a_{1L},a_{2L})$ non-degenerate, by Remark \ref{rem:suma}.
\end{proof}

\begin{cor} \label{cor:NNDq0} The following properties hold for the foliated surface $(\mathcal{M}',\mathcal{F}')$: 
	\begin{itemize}
		\item It is Newton non-degenerate at $q_0$ if and only if $(\mathcal{F},L)$ is non-degenerate for every $L\in\mathcal{S}_{p}^+$.
		\item It is Newton non-degenerate at $q_{\infty}$ if and only if $(\mathcal{F},L)$ is non-degenerate for every $L\in\mathcal{S}_{p}^-$. \qedhere
	\end{itemize}
\end{cor}
\begin{proof}
	Direct consequence of Lemma \ref{lema:NNDq0}.
\end{proof}

\begin{prop} \label{prop:estabilidadNND}
	Consider a CH-foliated surface $(\mathcal{M},\mathcal{F})$. Let $\pi:(\mathcal{M}',\mathcal{F'}) \rightarrow (\mathcal{M},\mathcal{F})$ be a combinatorial blowing-up centered at a point $p\in |M|$. We have that $(\mathcal{M},\mathcal{F})$ is Newton non-degenerate at $p$ if and only if $(\mathcal{M}',\mathcal{F'})$ is Newton non-degenerate at each point $q\in \pi^{-1}(p)$.
\end{prop}
\begin{proof}
	If $N=N^{\prec}_p(\mathcal{M},\mathcal{F})$ has no compact side with slope $-1$, we are done in view of Corollary \ref{cor:NNDq0} and Lemma \ref{lema:ladopendiente-1}. Assume now that $N$ has a compact side $L_0$ with slope $-1$. By Corollary \ref{cor:NNDq0}, it is enough to prove that $(\mathcal{F},L_0)$ is non-degenerate if and only if $(\mathcal{M}',\mathcal{F}')$ is Newton non-degenerate at every $q \in \mathcal{T}$, equivalently if and only if each $q \in \mathcal{T}$ is a presimple point. Let $\eta'$ be as in the proof of Lemma \ref{lema:lugarnopresimple}. Noting that $(\mathcal{F},L_0)$ is non-degenerate if and only if $(A_1,A_2)$ is non-degenerate, we are done as a result of Corollary \ref{cor:ladopendiente-1}.
\end{proof}
\begin{cor} \label{cor:estabilidad}
	Consider a CH-foliated surface $(\mathcal{M},\mathcal{F})$. Let $\pi:(\mathcal{M}',\mathcal{F'}) \rightarrow (\mathcal{M},\mathcal{F})$ be a finite composition of combinatorial blowing-ups.
	We have that $(\mathcal{M},\mathcal{F})$ is Newton non-degenerate at a given point $p \in |M|$ if and only if $(\mathcal{M}',\mathcal{F'})$ is Newton non-degenerate at each point $q\in \pi^{-1}(p)$.
\end{cor}
\subsection{Equivalence statement}
This subsection is devoted to prove the following result:
\begin{teo} \label{teo:equivalencia}
	A complex hyperbolic foliated surface $(\mathcal{M},\mathcal{F})$ is Newton non-degenerate if and only if it is of weak toric type.
\end{teo}

Let us do the proof at each point $p\in |M|$. 
\begin{enumerate}
	\item $e_p(E)\leq 1$. We have, by definition, that $(\mathcal{M},\mathcal{F})$ is Newton non-degenerate at $p$ if and only if $p$ is a presimple point. That happens if and only if $(\mathcal{M},\mathcal{F})$ is of weak toric type at $p$, by Remark \ref{rem:tipotoricoenpuntosnoesquina}.
	\item $e_p(E)=2$. Assume that $(\mathcal{M},\mathcal{F})$ is of weak toric type at $p$. This means that there is a finite composition of combinatorial blowing-ups $\pi: (\mathcal{M}',\mathcal{F}') \rightarrow (\mathcal{M}, \mathcal{F})$, such that each $q \in \pi^{-1}(p)$ is a presimple point. As a consequence, by definition $(\mathcal{M}',\mathcal{F}')$ is Newton non-degenerate for every $q \in \pi^{-1}(p)$. As a result of Corollary \ref{cor:estabilidad}, we obtain that $(\mathcal{M},\mathcal{F})$ is Newton non-degenerate at $p$. Conversely, assume that $(\mathcal{M},\mathcal{F})$ is Newton non-degenerate at $p$. Let us perform the combinatorial blowing-up $\pi: (\mathcal{M}',\mathcal{F}') \rightarrow (\mathcal{M}, \mathcal{F})$ with center $p$. We have that $(\mathcal{M}',\mathcal{F}')$ is Newton non-degenerate for each $q\in \pi^{-1}(p)$ by the stability property stated in Proposition \ref{prop:estabilidadNND}. As a consequence, the only non-presimple points $q \in \pi^{-1}(p)$ satisfy $e_q(E')=2$. The situation repeats, since $(\mathcal{M}',\mathcal{F}')$ is again Newton non-degenerate at these points. This allow us to conclude that the minimal pre-reduction of singularities of $(\mathcal{M}, \mathcal{F})_p$ is combinatorial and thus $(\mathcal{M}, \mathcal{F})$ is of weak toric type at $p$.
\end{enumerate}

\section{Non-degenerate Foliations on Projective Toric Surfaces} \label{sec:Newton Non-degenerate Foliations on Projective Toric Surfaces}
We recall the definition of projective toric surfaces and some results about their birational geometry. We introduce the homogeneous polygon $\Delta_h(\mathcal{F})$ of a foliation $\mathcal{F}$ on the complex projective plane $\mathbb{P}_{\mathbb{C}}^2$ and we prove that its area is zero, in the case of weak toric type complex hyperbolic foliations.
\subsection{Birational geometry of toric ambient surfaces}
The contents on this subsection can be essentially found at \cite{Ewa}. 

A \emph{toric surface} is an irreducible complex surface $S$ containing a two-dimensional complex torus $T\simeq (\mathbb{C}^*)^2$ as a Zariski open subset, such that the action of $T$ on itself extends to an algebraic action on $S$. The natural blowing-ups in the category of toric surfaces are the ones compatible with the torus action (\emph{equivariant}). This happens if and only if the center of the blowing-up is an orbit.

The union of the non-dense orbits of the torus action in a nonsingular toric surface $S$ is a strong normal crossings divisor $E_S$. We say that the pair $(S,E_S)$ is a \emph{toric ambient surface}. The points $p\in |S|$ with $e_p(E_{S})=2$ are exactly the closed orbits of the torus action.  As a consequence, in the category of toric surfaces, the equivariant blowing-ups are exactly the combinatorial ones.
\begin{rem}
	In view of the fact that a nonsingular toric surface $S$ gives in a natural way a toric ambient surface $(S,E_S)$, we use the expression \emph{foliation $\mathcal{F}$ on $S$} to make reference also to the foliated surface $((S,E_S),\mathcal{F})$.
\end{rem}
\begin{ejem}
	The first example of nonsingular projective toric surface is the projective plane $\mathbb{P}^2_\mathbb{C}$. The torus action is given in homogeneous coordinates $[X_0,X_1,X_2]$ by
	$$
	((t_1,t_2),[x_0,x_1,x_2]) \mapsto [x_0,t_1x_1,t_2x_2].
	$$
	The associated divisor is given by the three coordinate lines $X_0X_1X_2=0$. The standard affine example of nonsingular toric surface is $\mathbb{C}^2$. The torus action is given in coordinates $(x_1,x_2)$ by
	$$
	((t_1,t_2),(\alpha_1,\alpha_2)) \mapsto (t_1\alpha_1,t_2\alpha_2).
	$$
	The associated divisor is given by the coordinate lines $x_1x_2=0$. When we refer to $\mathbb{P}^2_\mathbb{C}$ or $\mathbb{C}^2$ as toric surfaces, we implicitly assume the above actions and coordinates.
\end{ejem}
The following result concerns the birational geometry of projective toric surfaces.
\begin{teo} [See \cite{Ewa}] \label{Teo:Oda}
	Given two nonsingular projective toric surfaces $S$ and $S'$, there is a nonsingular projective toric surface $S''$ and two finite sequences of equivariant blowing-ups $\pi:S'' \rightarrow S$, $\pi':S''\rightarrow S'$.
\end{teo}
\begin{cor}\label{cor:birracional}
	A given nonsingular projective toric surface $S$ is obtained from the toric surface $\mathbb{P}_{\mathbb{C}}^2$ by a finite sequence  $\mathbb{P}_{\mathbb{C}}^2 \rightarrow S$ of combinatorial blowing-ups and blowing-downs.
\end{cor}
\begin{proof}
	Take $S'= \mathbb{P}_{\mathbb{C}}^2$ in Theorem \ref{Teo:Oda}.
\end{proof}

\subsection{Newton non-degenerate foliations on $\mathbb{P}^2_\mathbb{C}$}
The aim of this subsection is to prove the following statement:
\begin{prop} \label{prop:area0}
	The homogeneous polygon $\Delta_h(\mathcal{F})$ of a CH-Newton non-degenerate foliation $\mathcal{F}$ on $\mathbb{P}_{\mathbb{C}}^2$ is a segment or a single point.
\end{prop}

\subsubsection{Foliations on the projective plane}
Before doing the proof, we recall basic definitions of foliations on the projective plane (see \cite{Can-C-D}) and we introduce the notion of homogeneous polygon. A foliation $\mathcal{F}$ on $\mathbb{P}_{\mathbb{C}}^2$ is given by a logarithmic homogeneous differential form
$$
W=A_0dX_0/X_0+A_1dX_1/X_1+A_2dX_2/X_2, \quad  A_i \in \mathbb{C}[X_0,X_1,X_2], \; i=0,1,2,
$$
where the coefficients $A_i$ are homogeneous polynomials of the same degree $d_{\mathcal{F}}$, without common factor and such that $A_0+A_1+A_2=0$. We say that $W$ is a \emph{homogeneous generator of $\mathcal{F}$}. If $W'$ is another homogeneous generator of $\mathcal{F}$, then $W'=kW$ with $k \in \mathbb{C}^*$ and conversely.
\begin{rem}
	The number $d_\mathcal{F}$ does not coincide with the so-called degree of the foliation. For instance, if the divisor $X_0X_1X_2=0$ has no dicritical components, the foliation is given by the holomorphic form
	$$
	X_0X_1X_2W =X_1X_2A_0dX_0+X_0X_2A_1dX_1+X_0X_1A_2dX_2,
	$$
	without common factors in the coefficients. Hence, the foliation degree is equal to $d_\mathcal{F}+1$. In a general way, the foliation degree is equal to $d_{\mathcal{F}}+1-\epsilon$, where $\epsilon$ is the number of dicritical components of $X_0X_1X_2=0$.
\end{rem}
\begin{defin}
	The \emph{homogeneous polygon $\Delta_h(\mathcal{F})$ of a foliation $\mathcal{F}$ on $\mathbb{P}_{\mathbb{C}^2}$} is given by
	$$
	\Delta_h(\mathcal{F})=\text{convex hull}\left(\Delta(A_0)\cup \Delta(A_1) \cup \Delta(A_2)\right) \subset \mathbb{R}^3,
	$$
	where $W=A_0dX_0/X_0+A_1dX_1/X_1+A_2dX_2/X_2$ is a homogeneous generator of $\mathcal{F}$.
\end{defin}
\begin{rem}
	Although $\Delta_h(\mathcal{F})$ is contained in $\mathbb{R}^3$, the name ``homogeneous polygon'' is due to the fact that $\Delta_h(\mathcal{F})\subset d_{\mathcal{F}}\Sigma^0_3$, where $\Sigma^0_3=\{(\sigma_0,\sigma_1,\sigma_2) \in \mathbb{R}_{\geq 0}^3 ; \, \sigma_0+\sigma_1+\sigma_2 = 1\}$. Note also that $\Delta_h(\mathcal{F})\cap (\sigma_i=0) \ne \emptyset$ for every $i=0,1,2$, since the coefficients $A_0,A_1,A_2$ have no common factors.
\end{rem}
We are interested in describing a foliation $\mathcal{F}$ of $\mathbb{P}_{\mathbb{C}}^2$ in terms of affine charts. We read the complex projective plane $\mathbb{P}_{\mathbb{C}}^2$ in affine charts $A_i=(X_i\ne 0) \subset \mathbb{P}_{\mathbb{C}}^2$, for $i=0,1,2$. We identify each $A_i$ with the affine toric variety $\mathbb{C}^2$ through the coordinates $(x_j^i,x_k^i)$ given by $x_j^i=X_j/X_i, \, x_k^i=X_k/X_i$, for $j, k \ne i$. Let us denote by $O_i$ to the origin of $A_i$. Note that $O_0=[1,0,0]$, $O_1=[0,1,0]$, $O_2=[0,0,1]$. We also denote by $D_i$ the divisor $E_{A_i}=(x_j^ix_k^i=0)$. We call \emph{affine $i$-chart $\mathcal{F}_i$ of $\mathcal{F}$} to the restriction $\mathcal{F}_i=\mathcal{F}|_{A_i}$. A generator of $\mathcal{F}_i$ adapted to the divisor $D_i$ is given by $\eta_i=a_j^idx_j^i/x_j^i+ a_k^idx_k^i/x_k^i$, for $\{i,j,k\}=\{0,1,2\}$, where $a_{\ell}^i=A_{\ell}(X_0/X_i,X_1/X_i,X_2/X_i)\in \mathbb{C}[x_j^i,x_k^i]$, for $\ell\in \{j,k\}$.
\begin{defin}
	An algebraic foliation $\mathcal{G}$ on the affine toric surface $\mathbb{C}^2$ is in \emph{general position} if the pair $(a_1,a_2)$ is in general position, where $\eta=a_1 dx_1/x_1+a_2 dx_2/x_2$, with $a_1,a_2 \in \mathbb{C}[x_1,x_2]$ is a generator of $\mathcal{G}$ adapted to the divisor $E_{\mathbb{C}^2}=(x_1x_2=0)$. The \emph{affine polygon $\Delta(\mathcal{G})$ of  $\mathcal{G}$} is defined by $\Delta(\mathcal{G})=\Delta(a_1,a_2)$.
\end{defin}
In particular, we have the affine polygons $\Delta(\mathcal{F}_i)$ associated to the affine $i$-charts of $\mathcal{F}$. The relationship between $\Delta_h(\mathcal{F})$ and $\Delta(\mathcal{F}_i)$ is given by $\Delta(\mathcal{F}_i)=\Phi_{i}^d(\Delta_h(\mathcal{F}))$, where $d=d_{\mathcal{F}}$ and $\Phi_i^d$ is the projection
$$
\Phi_{i}^d:d\Sigma_3^0 \rightarrow d\Sigma_2=\{(\sigma_1,\sigma_2) \in \mathbb{R}_{\geq 0}^2 ; \; \sigma_1+\sigma_2 \leq d\}
$$
defined by $\sigma =(\sigma_0,\sigma_1,\sigma_2) \mapsto (\sigma_j,\sigma_k)$, with $j,k\neq i$. We call to $\Delta(\mathcal{F}_i)$ the \emph{$i$-chart of $\Delta_h(\mathcal{F})$}. Given a side $L$ of $\Delta_h(\mathcal{F})$, the side $L_i=\Phi_{i}^d(L)$ of $\Delta(\mathcal{F}_i)$ is called the \emph{$i$-chart of $L$}.

Let us describe now the relationship between the homogeneous polygon $\Delta_h(\mathcal{G})$ and the Newton polygons $N_i=N^{\prec_i}_{\mathcal{O}_i}((A_i, D_i),\mathcal{F}_i)$, where the ordering $\prec_i$ is given by the natural order of the indices, that is $(x_j^i=0) \prec (x_k^i=0)$ if and only if $j> k$. Note that
$
N_i= \text{convex hull} (\Delta(\mathcal{F}_i)+\mathbb{R}_{\geq 0}^2),
$
then the compact sides of the Newton polygon $N_i$ are some of the sides of the affine polygon $\Delta(\mathcal{F}_i)$. Let $L$ be a side of $\Delta_h(\mathcal{F})$, we have that $L\subset \{\sigma\in d_{\mathcal{F}}\Sigma_3^0; \; \sigma_i=0\}$ if and only if $L_{\ell}$ is contained in a non-compact side of $N_{\ell}$ for ${\ell}=j,k.$ Otherwise, we have the following statement:
\begin{lema} \label{lema:poligonos}
	Given a side $L$ of $\Delta_h(\mathcal{F})$ such that $L\not\subset \{\sigma\in d_{\mathcal{F}}\Sigma_3^0; \; \sigma_{\ell}=0\}$ for any $\ell=0,1,2$, there is a unique affine chart $A_i$ such that $L_i$ is a compact side of $N_i$.
\end{lema}
\begin{proof}
	Consider the triangle given by $\Gamma_L^0= \{ \alpha(d,0,0)+(1-\alpha)\sigma; \; 0\leq \alpha \leq 1, \, \sigma \in L\}$. In the same way, define $\Gamma_L^1$ and $\Gamma_L^2$. We have that $\Gamma_L^i \cap  \Delta_h(\mathcal{F})=L$ and $\text{Ar}(\Gamma_L^i)\ne 0$ if and only if $L_i$ is a compact side of $N_i$. On the other hand, there is a unique $\ell \in \{0,1,2\}$ such that $\Gamma_L^{\ell} \cap \Delta_h(\mathcal{F})=L$ and $\text{Ar}(\Gamma_L^{\ell})\ne 0$, because of the convexity of $\Delta_h(\mathcal{F})$.
\end{proof}

\subsubsection{Proof of Proposition \ref{prop:area0}}
Let $\mathcal{F}$ be a foliation of the projective plane $\mathbb{P}_{\mathbb{C}}^2$ and let $W=A_0dX_0/X_0+A_1dX_1/X_1+A_2dX_2/X_2$ be a homogeneous generator of $\mathcal{F}$. Denote $d=d_{\mathcal{F}}$. Let us consider a side $L$ of the homogeneous polygon $\Delta_h(\mathcal{F})$. We give a description of the support-restrictions $A_{jL}$ of $A_j$ to $L$ as follows.

\begin{enumerate}
	\item When $L\subset \{\sigma\in d_{\mathcal{F}}\Sigma_3^0; \; \sigma_i=k\}$ for some $i \in \{0,1,2\}$ and $k\in \mathbb{Z}_{\geq 0}$. We detail the case $i=0$, the other ones are done in a similar way. The $A_{jL}$ are homogeneous polynomials of degree $d$ of the form $A_{jL}=X_0^k\tilde{A}_{jL}$, with $\tilde{A}_{jL}\in \mathbb{C}[X_1,X_2]$, for $j=0,1,2$. That is
	$$
	\tilde{A}_{jL}=X_1^{\sigma_1^j}X_2^{\sigma_2^j}\prod\nolimits_{\ell=1}^{d_j}(X_2-\alpha_{\ell j}X_1), \quad \sigma^j_1+\sigma^j_2+d_j=d-k.
	$$
	\item Otherwise, up to reordering of the variables, there is a monomial $X^{\sigma}$ such that for each $j=0,1,2$, we have $A_{jL}=X^{\sigma}\tilde{A}_{jL}$, where the $\tilde{A}_{jL} \in \mathbb{C}[U,V]$ are homogeneous polynomials of degree $n$ with $U=X_0^{\tilde{d}}$, $V=X_1^{\tilde{d}-\tilde{a}}X_2^{\tilde{a}}$, $\tilde{d}n=d-|\sigma|$ and $0 <\tilde{a}<\tilde{d}$. That is,
	$$
	\tilde{A}_{jL}=(X_0^{\tilde{d}})^{u_j}(X_1^{\tilde{d}-\tilde{a}}X_2^{\tilde{a}})^{v_j}\prod\nolimits_{\ell=1}^{n_j}(X_1^{\tilde{d}-\tilde{a}}X_2^{\tilde{a}}-\alpha_{\ell j}X_0^{\tilde{d}}); \quad u_j+v_j+n_j=n.
	$$
\end{enumerate}
\begin{rem} \label{rem:UV}
	The cases above can be considered in a unified way by writing $U=X_1,V=X_2$ when we are in the first one.
\end{rem}

\begin{lema} \label{lema:bastamirarenunacarta}
	Let $L$ be a side of $\Delta_h(\mathcal{F})$ and let $L_i$ be the $i$-chart of $L$. The following assertions are equivalent:
	\begin{enumerate}
		\item There is an index $i\in \{0,1,2\}$ such that $(\mathcal{F}^i,{L_i})$ is non-degenerate.
		\item For any index $i\in \{0,1,2\}$, we have that $(\mathcal{F}^i,{L_i})$ is non-degenerate. \qedhere
	\end{enumerate}
\end{lema}
\begin{proof}
	Take notations as in Remark \ref{rem:UV}. Consider $\{i,j,k\}=\{0,1,2\}$ and suppose that $(a_{jL_i}^i, a_{kL_i}^i)$ is degenerate, that is, they have a non-monomial common factor. Since
	$$
	a_{\ell L_j}^i=A_{\ell L}(X_0/X_i,X_1/X_i,X_2/X_i), \; \ell=j,k,
	$$
	the factor comes from a common factor $V-\alpha U$ of $A_{jL}, A_{kL}$. As a consequence $V-\alpha U$ also divides $A_{iL}$, since $A_{0L}+A_{1L}+A_{2L}=0$. Then $(a_{iL_j}^j, a_{kL_j}^j)$ and $(a_{iL_k}^k, a_{jL_k}^k)$ are also degenerate.
\end{proof}
We say that $(\mathcal{F},L)$ is \emph{non-degenerate} when the equivalent conditions of Lemma \ref{lema:bastamirarenunacarta} hold.
\begin{lema} \label{lema:generalposition}
	If $\mathcal{F}$ is a CH Newton non-degenerate foliation on $\mathbb{P}_{\mathbb{C}}^2$, then each affine $i$-chart $\mathcal{F}_i$ is in general position, for $i=0,1,2$.
\end{lema}
\begin{proof}
	In view of Lemma \ref{lema:bastamirarenunacarta}, it is enough to prove that given a side $L$ of $\Delta_h(\mathcal{F})$, there is an index $i \in \{0,1,2\}$ such that $(\mathcal{F}_i,L_i)$ is non-degenerate. Let us do it.
	
	If $L\subset \{\sigma \in d\Sigma_3^0; \; \sigma_1 =0\}$, then $L_0 \subset \{(\tau_1,\tau_2) \in \Sigma_2 ; \; \tau_1 =0\}$ and $a^0_{1L_0},a^0_{2L_0}$ are polynomials in the single variable $x_2^0$. Moreover, we can write
	$$
	a^0_1=a^0_{1L_0}+x_1^0\tilde{a}^0_1, \; a^0_2=a^0_{2L_0}+x_1^0\tilde{a}^0_2.
	$$
	If $(a^0_{1L_0},a^0_{2L_0})$ is degenerate, then there is $\lambda \in \mathbb{C}^*$ with $a^0_{1L_0}(\lambda)=a^0_{2L_0}(\lambda)=0$. This implies that the point $[1,0,\lambda] \in A_0$ is non-presimple for $\mathcal{F}$. This is impossible, since $\mathcal{F}$ is of weak toric type, by Theorem \ref{teo:equivalencia}. Then $(\mathcal{F}_0,{L_0})$ is non-degenerate. We reason in a similar way when $L\subset \{\sigma \in d\Sigma_3^0; \; \sigma_{\ell} =0\}$, for $\ell=0,2$.
	
	Otherwise, by Lemma \ref{lema:poligonos}, there is a unique $i \in \{0,1,2\}$ such that that $L_i$ is a compact side of $N^{\prec_i}_{O_i}((A_i, E_{A_i}),\mathcal{F}_i)$. Then, by definition, we have that $(\mathcal{F}_i,{L_i})$ is non-degenerate.
\end{proof}

\emph{Proof of Proposition \ref{prop:area0}}.
Let us work in the $0$-chart $\mathcal{F}_0$. In view of Lemma \ref{lema:generalposition}, we have that $\mathcal{F}_0$ is in general position. As a consequence, given a side $L_0$ of $\Delta(\mathcal{F}_0)=\Delta(a_1^0,a_2^0)$, we have that $(\mathcal{F}_0,{L_0})$ is non-degenerate and $\Delta(a_1^0)\cap L_0 \neq \emptyset$, $\Delta(a_2^0)\cap L_0 \neq \emptyset$, by Remark \ref{rem:una de las restricciones es cero}. The fact of being  $\mathcal{F}$ Newton non-degenerate also implies that $(a_1^0=a_2^0=0)\cap (\mathbb{C}^*)^2 = \emptyset.$
Applying Theorem \ref{teo:KKB} we conclude that $\text{MA}(\Delta(a_1^0),\Delta(a_2^0))=0$.
Now, by Remark \ref{rem:areamixta0}, we have two options, up to reordering, for $\Delta(a_1^0)$ and $\Delta(a_2^0)$:
\begin{enumerate}
	\item The affine polygon $\Delta(a_1^0)$ is a single point $\sigma$. In this case $\sigma$ belongs to each side of $\Delta(\mathcal{F}_0)$. This implies that $\Delta(\mathcal{F}_0)$ is a single point or a segment.
	\item The affine polygons $\Delta(a_1^0)$ and $\Delta(a_2^0)$ are parallel segments $L_1=\Delta(a_1^0)$, $L_2=\Delta(a_2^0)$. Recall that $\Delta(\mathcal{F}_0)$ is the convex hull of $L_1\cup L_2$. If $L_1 \cap L_2 \ne \emptyset$, then $\Delta(\mathcal{F}_0)$ is a segment. If $L_1 \cap L_2 = \emptyset$, then $\Delta(\mathcal{F}_0)$ has four sides where $L_1$ and $L_2$ are two of them. This contradicts the fact that $L_1=\Delta(a_1^0)$ intersects $L_2$.
\end{enumerate}
Thus, we have that $\Delta(\mathcal{F}_0)$ is a segment or a single point and the same happens with $\Delta_h(\mathcal{F})$.
\hfill{$\blacksquare$}

\section{Isolated Invariant Curves}
The main goal of this section is to prove that the isolated invariant branches of Newton non-degenerate foliations on projective toric ambient surfaces have a global nature. We also give local and global results about the existence of isolated invariant branches in the weak toric type and in the toric type contexts.

\subsection{Global nature of isolated invariant branches} \label{subsec:Global Nature of Isolated Invariant Branches}
The objective of this subsection is to prove the following result:
\begin{teo} \label{teo:prolongacion} 
	The isolated invariant branches of a CH-Newton non-degenerate foliation on a projective toric surface extend to projective algebraic curves.
\end{teo}
Since being Newton non-degenerate is equivalent to being of weak toric type in the complex hyperbolic frame, the previous result can be stated as follows:
\begin{quote}
	\emph{``The isolated invariant branches of a CH-weak toric type foliation on a projective toric surface extend to projective algebraic curves''.}
\end{quote}
Let $\mathcal{F}$ be a CH-Newton non-degenerate foliation on a projective toric surface $S$. In view of Theorem \ref{Teo:Oda}, there is a finite sequence of combinatorial blowing-ups and blowing-downs $\mathbb{P}_{\mathbb{C}}^2 \to S$. The transform $\mathcal{F}'$ of the foliation $\mathcal{F}$ by this sequence is a Newton non-degenerate foliation on $\mathbb{P}_{\mathbb{C}}^2$, by the stability property stated in Proposition \ref{prop:estabilidadNND}. If we prove that all the isolated invariant branches of $\mathcal{F}'$ extend to projective algebraic curves, we have proved also that the property holds for $\mathcal{F}$, because of the stability property of the isolated invariant branches stated in Proposition \ref{prop:estabilidadaisladas}.

As a consequence, it is enough to prove the result when $\mathcal{F}$ is defined on $\mathbb{P}^2_\mathbb{C}$. Recall that we have, up to reordering, three cases for the homogeneous polygon $\Delta_h(\mathcal{F})$:

\begin{enumerate}
	\item [a)] It is a single point.
	\item [b)] It is the segment joining the points $(0,d,0)$ and $(0,0,d)$, with $d > 0$.
	\item [c)] It is the segment joining the points $(d,0,0)$ and $(0,a,d-a)$, with $0 <a <d$.
\end{enumerate}
\paragraph{Case a).} There are no isolated invariant branches. Let us see it. A homogeneous generator $W$ is given by
$$
W=\lambda_0dX_0/X_0+\lambda_1dX_1/X_1+\lambda_2dX_2/X_2, \quad \lambda_0+\lambda_1+\lambda_2=0, \quad \lambda_i \in \mathbb{C},\; i=0,1,2.
$$
If $\lambda_0=0$, we have Sing$(\mathcal{F},E)=\text{Sing}(\mathcal{F})=\{O_0\}$ and $O_0$ is a corner presimple singularity without isolated invariant branches through it.

If $\lambda_0\lambda_1\lambda_2 \ne 0$, we have Sing$(\mathcal{F},E)=\text{Sing}(\mathcal{F})=\{O_0,O_1,O_2\}$. All the singularities are presimple corners, so there are no isolated invariant branches through them.

\paragraph{Case b).} The only isolated invariant branches are contained in a finite family of lines through $O_0$. Let us prove it. A homogeneous generator $W$ of $\mathcal{F}$ is given by
$$
W=A_0(X_1,X_2)dX_0/X_0+A_1(X_1,X_2)dX_1/X_1+A_2(X_1,X_2)dX_2/X_2, \quad A_0+A_1+A_2=0,
$$
where the $A_i$ are homogeneous of degree $d$. Let us consider the set $\mathcal{P}_{\Lambda_{\mathcal{F}}}= \{[0,1,\lambda] \in \mathbb{P}^2_\mathbb{C} ; \; \lambda\in \Lambda_{\mathcal{F}}\}$, where $\Lambda_{\mathcal{F}}=\{\lambda \in \mathbb{C}^*; \; A_0(1,\lambda)=0\}$. Note that:
\begin{enumerate}
	\item $X_0=0$ is invariant.
	\item $\mathcal{P}_{\Lambda_{\mathcal{F}}} \cup \{O_0\} \subset \text{Sing}(\mathcal{F},E) \subset \mathcal{P}_{\Lambda_{\mathcal{F}}} \cup \{O_0,O_1,O_2\}.$
	\item $\text{Sing}(\mathcal{F},E) \setminus \text{Sing}(\mathcal{F}) \subset \{O_0\}.$
\end{enumerate}

The point $O_i \in \text{Sing}(\mathcal{F})$ if and only if $X_j=0$ is invariant, for $\{i,j\}=\{1,2\}$. In this case, it is a corner type presimple point and there are no isolated invariant branches through it.

A point $P_\lambda \in \mathcal{P}_{\Lambda_{\mathcal{F}}}$ is a trace type presimple singularity. The germ at $P_{\lambda}$ of the line $\ell_\lambda=( X_2-\lambda X_1=0)$ is an invariant branch. By Lemma \ref{lema:aisladaspresimple}, there are no isolated branches through $P_\lambda$ different from $(\ell_\lambda, P_\lambda)$.

The point $O_0$ is non-presimple and it belongs to the lines $\ell_\lambda=( X_2-\lambda X_1=0)$. Let us prove that the isolated invariant branches at $O_0$ are among the germs $(\ell_\lambda,O_0)$, with $\lambda\in \Lambda_{\mathcal{F}}$. Let us work in the affine $0$-chart. A generator $\eta_0$ of $\mathcal{F}_0$ is given by
$$
\eta_0=A_1(x_1^0,x_2^0)dx_1^0/x_1^0+A_2(x_1^0,x_2^0)dx_2^0/x_2^0.
$$
The blowing-up at $O_0$ is determined, in the first chart, by the equations $x_1^0=u$, $x_2^0=uv$. Denote by $p_{0}$ the point $u=v=0$. A local generator at $p_{0}$ of the transform $\mathcal{F}'$ of $\mathcal{F}$ is $(A_1+A_2)(1,v)du/u+A_2(1,v)dv/v$. Note that $\{\lambda\in \mathbb{C}^* ; \; (A_1+A_2)(1,\lambda)=0\}=\Lambda_{\mathcal{F}}$, since $A_1+A_2=-A_0$. The points $p'_\lambda=(u=0,v=\lambda)$,
with $\lambda\in \Lambda_{\mathcal{F}}$ are trace type presimple singularities, since $\mathcal{F}$ is of weak toric type. The strict transform of the branches $(\ell_\lambda,O_0)$ with $\lambda \in \Lambda$ are the invariant branches $(\ell'_\lambda,p'_\lambda)$, where $\ell'_\lambda=(y_2-\lambda=0)$. On the other hand, we have that $p_{0}$ is a corner type presimple point, because of Remark \ref{rem:presimpleunsolovertice}. Analogously, the origin of the second chart $p_{\infty}$ is also a corner type presimple point. 
Hence, after blowing-up, the isolated invariant branches are among $(\ell'_\lambda,p'_\lambda)$ with $\lambda\in \Lambda_{\mathcal{F}}$. By the stability property established in Proposition \ref{prop:estabilidadaisladas}, there are no isolated invariant branches at $O_0$ different from $(\ell_\lambda,O_0)$ with $\lambda \in \Lambda_{\mathcal{F}}$.

\paragraph{Case c).} Let us prove that the only isolated invariant branches are contained in a finite family of curves of the type $X_1^{\tilde{d}-\tilde{a}}X_2^{\tilde{a}}-\lambda X_0^{\tilde{d}}=0$. Note that, they are locally cusps at the points $O_1$ and $O_2$. A homogeneous generator of $\mathcal{F}$ is given by
$$
W=A_0dX_0/X_0+A_1dX_1/X_1+A_2dX_2/X_2, \quad A_0+A_1+A_2=0, 
$$
where the $A_i$ belong to $\mathbb{C}[U,V]$, with $U=X_0^{\tilde{d}}$, $V=X_1^{\tilde{d}-\tilde{a}}X_2^{\tilde{a}}$, $d=\tilde{d}n$, $a=\tilde{a}n$ and $n= \text{gcd}(d,a)$. Note that:
\begin{enumerate}
	\item $X_1=0$, $X_2=0$ are not dicritical simultaneously.
	\item $\{O_1,O_2\} \subset \text{Sing}(\mathcal{F},E) \subset \{O_0,O_1,O_2\}$.
	\item $\text{Sing}(\mathcal{F},E) \setminus \text{Sing}(\mathcal{F}) \subset \{O_1,O_2\}.$
\end{enumerate}
The point $O_0 \in \text{Sing}(\mathcal{F})$ if and only if $X_1=0$, $X_2=0$ are both invariant. In this case, it is a corner type presimple singularity and there are no isolated invariant branches through it.

Let us consider the subset of $\mathbb{C}^*$ given by $\Lambda_{\mathcal{F}}=\{\lambda \in \mathbb{C}^*; \; (\tilde{a}A_0+\tilde{d}A_2)(1,\lambda)=0\}$ and the closed curves $\mathcal{C}_{\lambda}= (X_1^{\tilde{d}-\tilde{a}}X_2^{\tilde{a}}-\lambda X_0^{\tilde{d}}=0)$ with $\lambda \in \Lambda_{\mathcal{F}}$. Note that $O_1$, $O_2$ belong to $\mathcal{C}_\lambda$ and $(\mathcal{C}_{\lambda},O_1)$, $(\mathcal{C}_{\lambda},O_2)$ are invariant branches. Let us prove that the isolated invariant branches for $\mathcal{F}$ are among these ones. We do the proof at $O_1$ and a similar argument works at $O_2$. To do it, let us work in the affine $1$-chart. The point $O_1$ is a non-presimple point and a generator of $\mathcal{F}_1$ adapted to $D_1$ is given by
$$
\eta_1=A_0(x_0^1,1,x_2^1)dx_0^1/x_0^1+A_2(x_0^1,1,x_2^1)dx_2^1/x_2^1.
$$
Observe that $A_0(x_0^1,1,x_2^1)$ and $A_2(x_0^1,1,x_2^1)$ are quasi-homogeneous polynomials with weight vector $(\tilde{a},\tilde{d})$. Equivalently, the Newton polygon $N^{\prec_1}_{O_1}((A_1,D_1),\mathcal{F}_1)$ has exactly one compact side, whose slope is $-\tilde{a}/\tilde{d}$. In order to prove the result, we consider the composition
$$
((\tilde{A}_1,\tilde{D}_1), \tilde{\mathcal{F}}_1) \stackrel{\pi}\rightarrow (({A}'_1,{D}'_1), \mathcal{F}'_1) \stackrel{\sigma}\rightarrow ((A_1,D_1), \mathcal{F}_1)
$$
of the finite sequence of combinatorial blowing-ups that provides the minimal reduction of singularities of the cusp 
$
(x_2^1)^{\tilde{a}}-(x_0^1)^{\tilde{d}}=0,
$
where $\pi$ is the last blowing-up. Note that $e_{q}(E')=2$, where $q$ is the center of $\pi$. The morphism $\sigma$ is given in affine coordinates $(u,v)$ centered at $q$ by $x_0^1=u^\alpha v^\beta; \;x_2^1=u^\gamma v^\delta$, where $\alpha,\beta,\gamma,\delta$ is the only solution of the diophantine system
$$
\alpha+\beta=\tilde{a}; \; \gamma + \delta =\tilde{d},\quad \alpha\delta-\beta\gamma=1, \quad \alpha,\beta,\gamma,\delta \in \mathbb{Z}_{\geq 0}.
$$
A local generator $\eta_1'$ of $\mathcal{F}'_1$ adapted to $D_1'$ at $q$ is given by
$$
u^{\gamma a}v^{\beta d}\eta_1'=A_0(u^\alpha v^\beta,1,u^\gamma v^\delta)(\alpha du/u+ \beta dv/v)+A_2(u^\alpha v^\beta,1,u^\gamma v^\delta)(\gamma du/u+ \delta dv/v).
$$
We make the following remarks:
\begin{enumerate}
	\item The Newton polygon $N^{\prec'}_{q}((A'_1, D'_1), \mathcal{F}'_1)$ has exactly one compact side, which has slope $-1$.
	\item For each $s \in \text{Sing}(\mathcal{F}'_1,D'_1)$ we have $e_s(D'_1)=2$. This follows from Remark \ref{rem:ladopendiente-1} and the behaviour by blowing-up of the compact sides of the Newton polygon explained in Subsection \ref{subsec:poligonoexplosiones}.
	\item Each $s \in \text{Sing}(\mathcal{F}'_1,D'_1)\setminus \{q\}$ is a corner type presimple singularity, in view of Remark \ref{rem:presimpleunsolovertice}.
	\item The strict transform of the algebraic curve $\mathcal{C}_{\lambda}$ is the line $\ell_\lambda=(v-\lambda u=0)$.
\end{enumerate}

By Lemma \ref{lema:aisladaspresimple}, the strict transform of all the isolated invariant branches through $O_1$ passes through $q$. The, the problem is reduced to show that there are no isolated invariant branches at $q$ different from $(\ell_{\lambda},q)$ with $\lambda \in \Lambda_{\mathcal{F}}$. This is done by similar computations to the ones in case b).

\begin{rem}
	The set $\Lambda_\mathcal{F}$ is not empty. In fact, when $\Delta_h(\mathcal{F})$ is in case b) it has $d$ elements and when $\Delta_h(\mathcal{F})$ is in case c) it has $n$ elements. Let us prove the result for case b), the other case is done in a similar way, working after pre-reduction of singularities. We need to prove that $A_0$ has not multiple factors and also that $X_1$, $X_2$ do not divide $A_0$. Let $X_2-\lambda X_1$ be a multiple factor of $A_0$. A local generator of $\mathcal{F}$ adapted to $E_{\mathbb{P}_{\mathbb{C}}^2}$ at $P_{\lambda}$ is given by
	$$
	y^2\bar{A}_0(1,y+\lambda)dx_0^1/x_0^1+A_2(1,y+\lambda)dy, \quad A_2(1,\lambda)\ne 0, \quad A_0=(X_2-\lambda X_1)^2\bar{A}_0
	$$
	Hence $P_\lambda$ is a saddle-node, that is impossible since $\mathcal{F}$ is CH. Assume now that $A_0=X_1\bar{A_0}$. A local generator of $\mathcal{F}$ adapted to $E_{\mathbb{P}_{\mathbb{C}}^2}$ at $O_1$ is given by
	$$
	x_2^1\bar{A}_0(1,x_2^1)dx_0^1/x_0^1+A_2(1,x_2^1)dx_2^1/x_2^1, \quad A_2(1,0)\ne 0.
	$$
	We conclude that $O_1$ is a saddle-node, that can not hold. Analogously $X_2$ does not divide $A_0$.
\end{rem}

\subsection{Existence of isolated invariant branches}
We present a local result of existence of isolated invariant branches for toric type foliated surfaces. To do it, we need the following stability statement:
\begin{lema} \label{lema:noesquinapresimple}
	Let us consider a CH foliated surface $(\mathcal{M},\mathcal{F})$, a point $p\in \text{Sing}(\mathcal{F},E)$ and a combinatorial blowing-up $\pi:(\mathcal{M}',\mathcal{F}')\to (\mathcal{M},\mathcal{F})$. We have that $p$ is a presimple corner type point if and only if each point $p'\in \pi^{-1}(p) \cap \text{Sing}(\mathcal{F}',E')$ is presimple of corner type point.
\end{lema}
\begin{proof}
	It is enough to do the proof when $p$ is the center of blowing-up. 
	
	In Subsection \ref{sec:presimples}, we have proved that each point $p'\in \pi^{-1}(p) \cap \text{Sing}(\mathcal{F}',E')=\pi^{-1}(p) \cap\text{Sing}(\mathcal{F}')$ is a presimple singularity of corner type when $p$ is a corner type presimple point. Now, we have to prove that $p$ is a corner type presimple point assuming that each $p'\in \pi^{-1}(p) \cap \text{Sing}(\mathcal{F}',E')$ is a presimple of corner type. We distinguish two cases:
	
	The blowing-up is non-dicritical. Given $p'\in \pi^{-1}(p)$, we have that 
	$$
	\nu_{p'}(\mathcal{F}')=\mu_{p'}(\mathcal{F}')=
	\left\{
	\begin{array}{cl}
	0 & \text{ if } e_{p'}(E'_{\text{inv}})=1  \\
	1 & \text{ if } e_{p'}(E'_{\text{inv}})=2 \quad (\text{recall that } \mathcal{F}' \text{ is CH}),
	\end{array}
	\right.
	$$
	where $\nu_{p'}(\mathcal{F}')$ denotes the algebraic multiplicity of $\mathcal{F}'$ at $p'$ and $\mu_{p'}(\mathcal{F}')$ denotes the Milnor number (see \cite{Can-C-D}). Consider the Noether type formula (see \cite{Cam-N-S})
	$$
	\mu_p(\mathcal{F}) - \nu_p(\mathcal{F})^2 = S_p - (\nu_p(\mathcal{F}) +1 ) \geq 0 , \quad S_p=\sum\nolimits_{p' \in \pi^{-1}(p)}\mu_{p'}(\mathcal{F}').
	$$
	Thus, we have $1 \leq S_p \leq 2$. If $S_p=1$, then $\nu_p(\mathcal{F})=0$ and $e_p(E_{\text{inv}})=1$; this means that $p$ is a regular point and $\mathcal{F},E$ have normal crossings at $p$. If $S_p=2$, then $\nu_p(\mathcal{F})\in \{0,1\}$ and $e_p(E_{\text{inv}})=2$; thus, necessarily $\nu_p(\mathcal{F})= 1$. Observe that we have $\nu_p(\mathcal{F},E)=\nu_p(\mathcal{F}) + 1-e_p(E_{\text{inv}})=1+1-2=0$, then $p$ is a corner type presimple singularity for $(\mathcal{M},\mathcal{F})$. 
	
	The blowing-up is dicritical. We necessarily have $\text{Sing}(\mathcal{F}',E')=\emptyset$, this means that $p$ is singular,  $e_p(E_{\text{inv}})=2$ and there is a germ of vector field tangent to $\mathcal{F}$ whose linear part is the identity up to a factor (radial case). Then, $p$ is a presimple corner type singularity.
\end{proof}
\begin{prop} \label{prop:existenciaaislada}
	Assume that $(\mathcal{M},\mathcal{F})$ is of toric type at a point $p\in \text{Sing}(\mathcal{F},E)$. If $p$ is not a presimple point of corner type, there is an isolated invariant branch $(\Gamma,p)$ through it.
\end{prop}
\begin{proof}
	If $e_p(E)=1$, then $p$ is necessarily a trace type simple singularity and there is only an isolated invariant branch through $p$, in view of Remark \ref{rem:simpleaislada}. When $e_p(E)=2$, we consider the composition $\pi:(\mathcal{M}',\mathcal{F}')\to (\mathcal{M},\mathcal{F})$, given by a finite sequence of combinatorial blowing-ups, that induces a reduction of singularities over $p$. As a result of Lemma \ref{lema:noesquinapresimple}, there is a trace type simple singularity ${p'}\in\pi^{-1}(p)$ and an isolated invariant branch $(\Gamma',{p'})$ through $p'$. Then $(\Gamma,p)$ is an isolated invariant branch for $(\mathcal{M},\mathcal{F})$, where $(\Gamma,p)$ is the image of $(\Gamma',{p'})$.
\end{proof}

This proposition does not hold when $(\mathcal{M},\mathcal{F})$ is just of weak toric type at $p$: for instance, when $e_p(E)=1$ and the foliation is locally defined by the radial vector field. Nevertheless, when we work in a global way with weak toric type foliations on toric projective surfaces, we state in Proposition \ref{prop:existenciatoricodebil} a result of existence of isolated invariant branches.

Let $\mathcal{F}$ be a CH weak toric type foliation on $\mathbb{P}_{\mathbb{C}}^2$. We take notations as in Subsection \ref{subsec:Global Nature of Isolated Invariant Branches}, recalling, in particular the existence of the cases a), b) and c) for  the homogeneous polygon $\Delta_h(\mathcal{F})$.

Assume that $\Delta_h(\mathcal{F})$ is not a single point. Consider $\lambda\in \Lambda_{\mathcal{F}}$ and denote 
\begin{eqnarray*}
	Y=\ell_\lambda, \; P_1=P_{\lambda} \text{ and } P_2=O_0, \text{ if } \Delta_h(\mathcal{F}) \text{ is in case b)}. \\
	Y=\mathcal{C}_{\lambda}, \; P_1=O_1 \text{ and } P_2=O_2, \text{ if } \Delta_h(\mathcal{F}) \text{ is in case c)}.
\end{eqnarray*}
We have that $Y\cap E= \{P_1,P_2\}$ and $P_1,P_2$ are not presimple corner type points of $\text{Sing}(\mathcal{F},E)$. Moreover, the germs $(Y,P_1)$ and $(Y,P_2)$ are irreducible branches: in case b) they are lines and in case c) they are cusps of types $(\tilde{d},\tilde{d}-\tilde{a})$ and $(\tilde{d},\tilde{a})$, respectively. Let
$$
\pi:((M',E'),\mathcal{F}')\to ((\mathbb{P}_{\mathbb{C}}^2 ,E_{\mathbb{P}_{\mathbb{C}}^2 }),\mathcal{F})
$$
be a combinatorial pre-reduction of singularities and denote by $Y'$ the strict transform of $Y$ by $\pi$. We have that $Y'\cap E'=\{p_1',p_2'\}$, where each $p_i'\in \pi^{-1}(P_i)$ is a trace type presimple singularity, for $i=1,2$. Let $\xi_i$ be a germ of vector field at $p'_i$ tangent to $\mathcal{F'}$ with non-nilpotent linear part $L_{\xi_i}$ and denote by $r(L_{\xi_i})=\{\alpha_i,1/\alpha_i\}$ the ratios of its eigenvalues.
\begin{lema} \label{lema:conexion}
	We have $r(L_{\xi_1})=-r(L_{\xi_2})$. In particular, there is a simple singularity between $p'_1$ and $p'_2$.
\end{lema}
\begin{proof}
	In case b), the eigenvalues of $L_{\xi_1}$ are $\bar{A}_0(1,\lambda)$ and $-A_2(1,\lambda)$, where $A_0=(X_2-\lambda X_1)\bar{A}_0$. The eigenvalues of $L_{\xi_2}$ are $-\bar{A}_0(1,\lambda)$ and $-A_2(1,\lambda)$, since $A_1+A_2=-A_0$. 
	
	In case c), we start by considering the sequence of combinatorial blowing-ups that provides the minimal reduction of singularities of the cusp $(Y,P_1)$. It is given in affine coordinates $(u_1,v_1)$ centered at $p_1'$ by $x_0^1=u_1^{\tilde{a}}(v_1+\lambda)^{\beta}$ and $x_2^1=u_1^{\tilde{d}}(v_1+\lambda)^{\delta}$. A local generator of the transform of $\mathcal{F}$ adapted to the divisor $(u_1=0)$ is given by
	$$
	(\tilde{a}A_0+\tilde{d}A_2)(1,v_1+\lambda)\frac{du_1}{u_1}+\frac{(\beta A_0+\delta A_2)(1,v_1+\lambda)}{v_1+\lambda}dv_1,
	$$
	and the eigenvalues of $L_{\xi_1}$ are 
	$$
	\mu_1=-\frac{(\beta A_0+\delta A_2)(1,\lambda)}{\lambda}, \quad \rho_1=B(1,\lambda); \quad B(U,V)=\frac{\tilde{a}A_0+\tilde{d}A_2}{V-\lambda U}.
	$$
	Now, we consider the sequence of combinatorial blowing-ups obtained following the infinitely near points of the cusp $(Y,P_2)$. It is given in affine coordinates $(u_2,v_2)$ centered at $p_2'$ by $x_0^2=u_2^{\tilde{d}-\tilde{a}}(v_2+1/\lambda)^{\delta-\beta}$ and $x_1^2=u_2^{\tilde{d}}(v_2+1/\lambda)^{\delta}$.
	A local generator of the transform of $\mathcal{F}$ adapted to the divisor $(u_2=0)$ is given by
	$$
	((\tilde{d}-\tilde{a})A_0+\tilde{d}A_1)(v_2+1/\lambda,1)\frac{du_2}{u_2}+\frac{((\delta-\beta)A_0+\delta A_1)(v_2+1/\lambda,1)}{v_2+1/\lambda}dv_2,
	$$
	and the eigenvalues of $L_{\xi_2}$ are 
	$$
	\mu_2=-\frac{((\delta-\beta)A_0+\delta A_1)(1/\lambda,1)}{\frac{1}{\lambda}} \quad \rho_2=C(1/\lambda,1); \quad C(U,V)=\frac{(\tilde{d}-\tilde{a})A_0+\tilde{d}A_1}{-\frac{1}{\lambda}(V-\lambda U)}.
	$$
	We have the relations $\mu_1=-\lambda^{n-2}\mu_2$ and $\rho_1=\lambda^{n-2}\rho_2$, since $A_1=-(A_0+A_2)$. Hence $r(L_{\xi_1})=-r(L_{\xi_2})$.
\end{proof}
\begin{rem} \label{rem:todossonsimples}
	When $\mathcal{F}$ is of toric type, $p'_1$ and $p'_2$ are both trace type simple singularities. Hence $(Y,P_1)$ and $(Y,P_2)$ are isolated invariant branches. 
\end{rem}

\begin{prop} \label{prop:existenciatoricodebil}
	Let $\mathcal{F}$ be a weak toric type foliation on a projective toric surface $S$. Assume that there is a point $p\in \text{Sing}(\mathcal{F},E)$ which is not presimple of corner type. We have that there is an isolated invariant branch $(\Gamma,q)$ passing through some $q\in \text{Sing}(\mathcal{F},E_S)$.
\end{prop}
\begin{proof}
	It is enough to do the proof when $S=\mathbb{P}_{\mathbb{C}}^2$, because of Lemma \ref{lema:noesquinapresimple} and the stability of the isolated invariant branches by combinatorial blowing-ups and blowing-downs stated in Proposition \ref{prop:estabilidadaisladas}. When the homogeneous polygon $\Delta_h(\mathcal{F})$ has a single vertex, all the points in $\text{Sing}(\mathcal{F},E_{\mathbb{P}_{\mathbb{C}}^2})$ are presimple of corner type. As a consequence $\Delta_h(\mathcal{F})$ is not a single vertex and it is in case b) or in case c). By Lemma \ref{lema:conexion}, there is a trace type simple singularity $q'$ after performing a pre-reduction of singularities that we denote by $\pi$. The (only) invariant branch $(Y',{q'})$ through $q'$ provides an isolated invariant branch $(Y,q)$ through $q=\pi(q')\in\text{Sing}(\mathcal{F},E_{\mathbb{P}_{\mathbb{C}}^2})$.
\end{proof}

When we are in the toric type case, we can precise the above statement as follows:
\begin{prop}
	Let us consider a toric type foliation $\mathcal{F}$ on a projective toric surface $S$ and an isolated invariant branch $(\Gamma,p)$. Let $Y$ be a projective algebraic curve extending $(\Gamma,p)$. We have that any branch $(\Upsilon,q)\subset (Y,q)$ with $q\in Y \cap E$ is an isolated invariant branch.
\end{prop}
\begin{proof}
	It follows from Theorem \ref{teo:prolongacion} and Remark \ref{rem:todossonsimples}.
\end{proof}
\section{Rational First Integrals and Global Invariant Curves}

We prove the following property for toric type foliations on projective toric surfaces: Either there are finitely many global invariant curves different from the divisor, all of them extending isolated invariant branches, or there are infinite many global invariant curves (there is rational first integral), but there are no isolated invariant branches.

We say that an invariant branch $(\Gamma,p)\not\subset (E,p)$ of a foliated surface $(\mathcal{M},\mathcal{F})$ is \emph{proper for $(\mathcal{M},\mathcal{F})$} when $p\in E$. This property is stable by combinatorial blowing-ups and blowing-downs. When $(\mathcal{M},\mathcal{F})$ is of weak toric type, all the isolated invariant branches are proper, since every singularity belongs to the divisor.

\begin{teo} \label{teo:dicotomia}
	We have the following dichotomy for a toric type foliation $\mathcal{F}$ on a projective toric surface $S$:
	\begin{enumerate}
		\item [I)] There is rational first integral and there are no isolated invariant branches.
		\item [II)] There is no rational first integral and every proper invariant branch extending to a projective algebraic curve is an isolated invariant branch. \qedhere
	\end{enumerate}
\end{teo}
It is enough to prove the dichotomy for the case $S=\mathbb{P}_{\mathbb{C}}^2$. Indeed, having rational first integral, being a projective algebraic curve, being invariant, proper and isolated are properties that have a good behaviour under blowing-ups and blowing-downs. 

\begin{rem}
	Every global invariant curve $Y$ not contained in the divisor extends at least a proper branch. To see it, just note that we have $Y \cap (X_i=0) \ne \emptyset$, when $S=\mathbb{P}_{\mathbb{C}}^2$.
\end{rem}
Assume $S=\mathbb{P}_{\mathbb{C}}^2$. We distinguish between the cases a), b) and c) for the homogeneous polygon $\Delta_h(\mathcal{F})$ and we take notations as in Subsection \ref{subsec:Global Nature of Isolated Invariant Branches}.
\paragraph{Case a).} There are no isolated invariant branches for $\mathcal{F}$. We have two options:
\begin{itemize}
	\item The points $O_0,O_1,O_2$ are all simple corners. The components of the divisor $X_i=0$ are invariant, for $i=0,1,2$. Assume that there is a global invariant curve $Y$ different from these ones. We have that $Y$ cuts $X_0=0$ at singular points, that are necessarily $O_1$ or $O_2$. This contradicts the fact that $O_1$ and $O_2$ are simple singularities. As a consequence, such a $Y$ can not exist. In particular, there is no rational first integral for $\mathcal{F}$ and we are in situation II).
	\item One of the $O_i$ is presimple but non-simple, say $O_0$. We have that 
	$$
	\lambda_0=m-n, \lambda_1=n, \lambda_2=-m, \; n,m\in \mathbb{Z}_{>0}.
	$$
	Then $d(X_0^{m-n}X_1^{n}X_2^{-m})$ is a rational first integral and we are in situation I). 
\end{itemize}
\paragraph{Case b).} 
Let us see that situation II) holds. By Remark \ref{rem:todossonsimples}, we have that $(\ell_\lambda,O_0)$, $(\ell_\lambda,P_\lambda)$ are isolated invariant branches for every $\lambda \in \Lambda_{\mathcal{F}}$ and they are the only ones. Then, it just remains to prove that a global invariant curve is either a component of the divisor or a line $\ell_\lambda$ with $\lambda \in \Lambda_{\mathcal{F}}$. 

Let $Y$ be a global irreducible invariant curve of degree $r$ different from $\ell_{\lambda_0}$ for some $\lambda_0 \in \Lambda_{\mathcal{F}}$. The curve $Y$ intersects $\ell_{\lambda_0}$ in $r$ singular points (counted with multiplicity). They are necessarily $O_0$ or $P_{\lambda_0}$ and in fact, just $O_0$, since $P_{\lambda_0}$ is a simple singularity. If $Y$ is a line, it cuts transversally $\ell_{\lambda_0}$ at $O_0$ and it intersects $X_0=0$ in a single singularity. As a consequence it is $X_1=0$, $X_2=0$ or a line $\ell_\lambda$ with $\lambda \in \Lambda_{\mathcal{F}} \setminus \{\lambda_0\}$. Otherwise, when $r > 1$, we have that $Y$ and $\ell_{\lambda_0}$ have the same tangent at $O_0$. When we perform the blowing-up, we obtain three invariant branches passing through $p'_{\lambda_0}$: the exceptional divisor, $\ell'_{\lambda_0}$ and the strict transform $Y'$ of $Y$. This contradicts the fact that $p'_{\lambda_0}$ is a simple singularity.
\paragraph{Case c).} 
Let us prove that situation II) holds. Like in case b), the problem is reduced to show that a global invariant curve is either a component of the divisor or one of the curves $\mathcal{C}^{\mathcal{F}}_\lambda=\mathcal{C}_\lambda$ with $\lambda \in \Lambda_{\mathcal{F}}$.

We take the rational map $\varphi:\mathbb{P}_{\mathbb{C}}^2 \setminus{[0,0,1]} \rightarrow \mathbb{P}_{\mathbb{C}}^2$ given in coordinates by
$$
Y_0=X_0^{\tilde{d}}, \; Y_1=X_1^{\tilde{d}-\tilde{a}}X_2^{\tilde{a}}, \; Y_2=X_2^{\tilde{d}}.
$$
This map is compatible with the torus action. A homogeneous generator $W$ of $\mathcal{F}$ is given by
$$
W=A_0(Y_0,Y_1)dX_0/X_0+A_1(Y_0,Y_1)dX_1/X_1+A_2(Y_0,Y_1)dX_2/X_2,
$$
where the coefficients $A_i$ are homogeneous polynomials of degree $n$ and $A_0+A_1+A_2=0$. We have that $\mathcal{F}$ is the pull-back by $\varphi$ of a the toric type foliation $\mathcal{G}$ of $\mathbb{P}_{\mathbb{C}}^2$ generated by
$$
\Omega=(\tilde{d}-\tilde{a})A_0dY_0/Y_0+\tilde{d}A_1dY_1/Y_1+(-\tilde{a}A_1+ (\tilde{d}-\tilde{a})A_2)dY_2/Y_2.
$$
The homogeneous polygon $\Delta_h(\mathcal{G})$ is in case b). Moreover, we have
$$
\Lambda=\Lambda_{\mathcal{G}}=\{\lambda \in \mathbb{C}^*; \; (-\tilde{a}A_1+ (\tilde{d}-\tilde{a})A_2)(1,\lambda)=0\}=\{\lambda \in \mathbb{C}^*; \; (\tilde{a}A_0 + \tilde{d}A_2)(1,\lambda)=0\}=\Lambda_{\mathcal{F}}.
$$
As a consequence, there is a biunivocal relation between $\mathcal{C}_{\lambda}^{\mathcal{F}}$ and $\ell_{\lambda}^{\mathcal{G}}$. More precisely 
$$
\varphi(\mathcal{C}_{\lambda}^{\mathcal{F}})=\ell_{\lambda}^{\mathcal{G}}, \quad \varphi^{-1}(\ell_{\lambda}^{\mathcal{G}})=\mathcal{C}_{\lambda}^{\mathcal{F}}.
$$
Let $Y\not\subset (X_0X_1X_2=0)$ be a global curve invariant for $\mathcal{F}$. 
We have that $\varphi(Y) \not\subset(Y_0Y_1Y_2=0)$ and it is a global curve invariant for $\mathcal{G}$. By the study of case b), there is a $\lambda\in \Lambda$ such that $\varphi(Y)=\ell_{\lambda}^{\mathcal{G}}$ and hence $Y=\mathcal{C}_{\lambda}^{\mathcal{F}}$.
\fun
The author is partially supported by the Ministerio de Educaci\'on, Cultura y Deporte of Spain (FPU14/02653 grant) and by the Ministerio de Econom\'ia y Competitividad from Spain, under the Project ``Algebra y geometr\'ia en sistemas din\'amicos y foliaciones singulares.'' (Ref.:  MTM2016-77642-C2-1-P).
\acks
I would like to express my deepest gratitude to Professor Felipe Cano for his unconditional supervision of this work. I would also like to thank Professor Clementa Alonso for all her suggestions.

\end{document}